%% file: OnlineDirectedHamilton_arxiv.tex
\documentclass[11pt]{article}

\usepackage{amsmath, amsthm, amssymb}
\usepackage{enumerate}

\usepackage{tikz}
\usetikzlibrary{arrows}

\date{}

\title{\vspace{-0.8cm}Getting a directed Hamilton cycle two times faster}
\author{
Choongbum Lee \thanks{Department of Mathematics, UCLA, Los
Angeles, CA, 90095. Email: choongbum.lee@gmail.com. Research
supported in part by Samsung Scholarship.}
\and
Benny Sudakov \thanks{Department of Mathematics, UCLA, Los Angeles, CA 90095.
Email: bsudakov@math.ucla.edu. Research supported in part by NSF grant
DMS-1101185, NSF CAREER award DMS-0812005 and by USA-Israeli BSF grant. }
\and
Dan Vilenchik \thanks{Department of Mathematics, UCLA, Los Angeles, CA 90095.
Email: vilenchik@math.ucla.edu.
}
}

\theoremstyle{plain}
\newtheorem{THM}{Theorem}[section]
\newtheorem{PROP}[THM]{Proposition}
\newtheorem{LEMMA}[THM]{Lemma}
\newtheorem{DEF}[THM]{Definition}
\newtheorem{COR}[THM]{Corollary}
\newtheorem{CLAIM}[THM]{Claim}

\theoremstyle{definition}

\newcommand{\BBE}{\mathbb{E}}
\newcommand{\BFP}{\mathbb{P}}

\newcommand{\ALG}{{\bf Orient}}
\newcommand{\ALGPRIME}{{\bf OrientPrime}}
\newcommand{\BIP}{{\bf BIP}}
\newcommand{\FIVEINOUT}{D_{5-in,5-out}}

\newcommand{\VTXONEDEG}[2]{d_1(#1)}
\newcommand{\VTXFREETWODEG}[2]{d_{2}(#1)}   %{d_{2s}(#1)}
\newcommand{\PHASETWODEG}[2]{d_{AB}(#1)}
\newcommand{\OUT}{\textrm{\textit{OUT}}}
\newcommand{\IN}{\textrm{\textit{IN}}}

\newcommand{\RV}{RV}

\newcommand{\HAM}{\mathcal{HAM}}
\newcommand{\whp}{whp}

\newcommand{\Erdos}{Erd\H{o}s}
\newcommand{\Renyi}{R\'enyi}

\newcommand{\Bollobas}{Bollob\'as}

\newcommand{\Komlos}{Koml{\'o}s}

\newcommand{\Szemeredi}{Szemer\'edi}

\begin{document}
\maketitle

\begin{abstract}
Consider the random graph process where we start with an empty graph
on $n$ vertices, and at time $t$, are given an edge $e_t$ chosen
uniformly at random among the edges which have not appeared so far.
A classical result in random graph theory asserts that $\whp$ the
graph becomes Hamiltonian at time $(1/2+o(1))n \log n$. On the
contrary, if all the edges were directed randomly, then the graph has a
directed Hamilton cycle $\whp$ only at time $(1+o(1))n \log n$. In
this paper we further study the directed case, and ask whether it is
essential to have twice as many edges compared to the undirected
case. More precisely, we ask if at time $t$, instead of a random
direction one is allowed to choose the orientation of $e_t$, then
whether it is possible or not to make the resulting directed graph Hamiltonian
at time earlier than $n \log n$. The main result of our paper
answers this question in the strongest possible way, by asserting that one
can orient the edges on-line so that $\whp$, the resulting graph has
a directed Hamilton cycle exactly at the time at which the
underlying graph is Hamiltonian.
\end{abstract}

\section{Introduction}

The celebrated \emph{random graph process}, introduced by \Erdos~and
\Renyi~\cite{ErdRen} in the 1960's, begins with an empty graph on
$n$ vertices, and in every round $t=1,\ldots,m$ adds to the
current graph a single new
edge chosen uniformly at random out of all missing edges.
This distribution is commonly denoted as $G_{n,m}$.
An equivalent ``static" way of defining $G_{n,m}$ would be: choose
$m$ edges uniformly at random out of all $\binom{n}{2}$ possible
ones. One advantage in studying the random graph process, rather
than the static model, is that it allows for a higher resolution
analysis of the appearance of monotone graph properties (a graph
property is monotone if it is closed under edge addition).
%One such property is the existence of a Hamilton cycle.

A \emph{Hamilton cycle} of a graph is a simple cycle that passes through
every vertex of the graph, and a graph containing a Hamilton cycle
is called \emph{Hamiltonian}. Hamiltonicity is one of the most fundamental
notions in graph theory, and has been intensively studied in various
contexts, including random graphs. The earlier results on
Hamiltonicity of random graphs were obtained by P\'osa~\cite{Posa},
and Korshunov \cite{Korshunov}. Improving on these results,
\Komlos~and \Szemeredi~\cite{KomSze} proved that if $m' =
\frac{1}{2}n \log n + \frac{1}{2} \log \log n + c_n n$, then
\[ \lim_{n \rightarrow \infty} \BFP(G_{n,m'} \, \textrm{is Hamiltonian}) =
 \left\{ \begin{array}{cll}
  & 0                   & \textrm{if }c_n \rightarrow -\infty \\
  & e^{-e^{-2c}} &  \textrm{if }c_n \rightarrow c \\
  & 1                   &   \textrm{if }c_n \rightarrow \infty.
  \end{array}\right.
 \]
One obvious necessary condition for the graph to be Hamiltonian is
for the minimum degree to be at least 2, and surprisingly, the
probability of $G_{n,m'}$ having minimum degree two at time ${m'}$
has the same asymptotic behavior as the probability of it being
Hamiltonian. \Bollobas~\cite{Bollobas} strengthened this observation
by proving that $\whp$ the random graph process becomes Hamiltonian when the last
vertex of degree one disappears. Moreover, \Bollobas, Fenner, and
Frieze \cite{BolFenFri} described a polynomial time algorithm which
$\whp$ finds a Hamilton cycle in random graphs.

Hamiltonicity has been studied for directed graphs as well.
Consider a {\em random directed graph process} where at time $t$ a
random directed edge is chosen uniformly at random among all missing
edges. and let $D_{n,m}$ be the graph consisting of the first $m$
edges. Frieze \cite{Frieze} proved that for $m'' = n \log n + c_n
n$, the probability of $D_{n,m''}$ containing a (directed) Hamilton
cycle is
\[ \lim_{n \rightarrow \infty} \BFP(D_{n,m''} \, \textrm{is Hamiltonian}) =
 \left\{ \begin{array}{cll}
  & 0                   & \textrm{if }c_n \rightarrow -\infty \\
  & e^{-2e^{-c}} &  \textrm{if }c_n \rightarrow c \\
  & 1                   &   \textrm{if }c_n \rightarrow \infty.
  \end{array}\right.
 \]
Similar to the undirected case, this probability has the same
asymptotic behavior as the probability of the directed graph having
minimum in-degree and out-degree 1. In fact, Frieze proved
\cite{Frieze} that when the last vertex to have in- or out-degree
less than one disappears, the graph has a Hamilton cycle $\whp$.

Hamiltonicity of various other random graph models has also been
studied \cite{RobWor, BalBolKriMulWal}. One model which will be of
particular interest to us is the $k$-in $k$-out model, in which every vertex
chooses $k$ in-neighbors and $k$-out neighbors uniformly at random
and independently of the others. Improving on several previous
results, Cooper and Frieze \cite{CooFri} proved that a random graph
in this model is Hamiltonian $\whp$ already when $k=2$ (which is
best possible since it is easy to see that a 1-in 1-out random graph
is $\whp$ not Hamiltonian).

\subsection{Our Contribution}

\Bollobas~\cite{Bollobas}, and Frieze's \cite{Frieze} results
introduced above suggest that the main obstacle to
Hamiltonicity of random graphs lies in ``reaching'' certain minimum
degree conditions. It is therefore natural to ask how the thresholds
change if we modify the random graph process so
that we can somehow bypass this obstacle.

%In this context, we consider the following process.
% The question that we study
%in this paper is whether one really ``needs" $(1+o(1))n\log n$
%random edges to obtain a directed Hamilton cycle.
We consider the following process suggested by Frieze \cite{Frieze2}
which has been designed for this purpose.
Starting from the empty graph, at time $t$, an undirected edge
$(u,v)$ is given uniformly at random out of all missing edges, and a
choice of its orientation ($u \to v$ or $v \to u$) is to be made at
the time of its arrival. In this process, one can
attempt to accelerate the appearance of monotone directed graph
properties, or delay them, by applying an appropriate on-line algorithm.
It is important to stress that the process
is {\em on-line} in nature, namely, one cannot see
any future edges at the current round and is forced to make the
choice based only on the edges seen so far.
In this paper, we investigate the property of containing a directed
Hamilton cycle by asking the question, ``can one speed up the appearance
of a directed Hamilton cycle?''. The best we can hope for is to
obtain a directed Hamilton cycle at the time when the underlying
graph has minimum degree 2. The following result asserts that
directed Hamiltonicity is in fact achievable exactly at that time,
and this answers the above question positively in the
strongest possible way.

\begin{THM} \label{thm_mainthm2}
Let $\mathcal{G}$ be a random (undirected) graph process
that terminates when the last vertex of degree one disappears.
There exists an on-line algorithm $\ALG$ that orients the edges of
$\mathcal{G}$, so that the resulting directed graph is Hamiltonian
$\whp$.
\end{THM}

%Theorem \ref{thm_mainthm2} is clearly best possible, since the
%underlying graph should have minimum degree at least 2 in order to
%have an orientation with a directed Hamilton cycle.
Let us remark that $\mathcal{G}$ $\whp$ contains $(1+o(1))n \log
n/2$ edges, in contrast with $(1+o(1))n \log n$ edges in the random
directed graph model. Thus the required number of random edges is
reduced by half.

%Theorem \ref{thm_mainthm2} also supports the principal mentioned at
%the beginning of this subsection, that the minimum degree is the
%main obstacle of a random graph being Hamiltonian.
%Thus as in the previous
%works on Hamiltonicity of random graphs, we are able to tie the
%minimum degree of the underlying random graph with the possibility
%of an orientation that induces a Hamilton cycle.
%This also answers a question of Frieze \cite{Frieze2}.

Our model is similar in spirit to the so called Achlioptas process.
It is well known that a giant connected component (i.e. a component
of linear size) appears in the random graph
$G_{n,m}$ when $m=(1+o(1))n/2$. Inspired by the celebrated ``power
of two choices" result \cite{Azar94balancedallocations},
%($n$ balls are randomly thrown into $n$ bins but this time every ball chooses a bin out of two possible random bins)
Achlioptas posed the following question: Suppose that edges arrive
in pairs, that is in round $t$ the pair of edges $(e_t,e'_t)$ chosen
uniformly at random is given, and one is allowed to pick an edge out
of it for the graph (the other edge will be discarded). Can one
delay the appearance of the giant component? Bohman and Frieze
answered this question positively \cite{bohman01avoidingGiant} by
describing an algorithm whose choice rule allows for the ratio $m/n
\geq 0.53$, and this ratio has been improved since \cite{BohFriWor}.
Quite a few papers have thereafter studied various related problems that arise in
the above model \cite{BohKra, FlaGamSor, KriLohSud, SinVil, SpeWor}.
As an example, in \cite{KriLohSud}, the authors studied the question,
``How long can one delay the appearance of a certain fixed
subgraph?''.

One such paper which is closely related to our work
is the recent work of Krivelevich, Lubetzky, and
Sudakov \cite{KriLubSud}. They studied the Achlioptas
process for Hamiltonicity, and proved that by exploiting the ``power
of two choices'', one can construct a Hamilton cycle at time
$(1+o(1))n \log n /4$, which is twice as fast as in the random case.
Both our result and this result suggest that the ``bottleneck'' to
Hamiltonicity of random graphs indeed lies in the minimum degree, and
thus these results can be understood in the context of complementing
the results of
\Bollobas~\cite{Bollobas}, and Frieze \cite{Frieze}.
%, in the ``classical" form of the process. Namely,
%For a fixed integer $k>0$, at time $i$, $k$ random edges are offered, and one chooses one of them to be %included in the graph.
%This work is in some sense orthogonal to the current one, as the two settings seem to be technically unrelated.

\subsection{Preliminaries}

The paper is rather involved technically. One factor that
contributes to this is the fact that we are establishing the
``hitting time" version of the problem. That is, we determine the
exact threshold for the appearance of a Hamilton cycle. The analysis
can be simplified if one only wishes to estimate this threshold
asymptotically (see concluding remarks). To make the current
analysis more approachable without risking any significant
change to the random model, we consider the following variant of the
graph process, which we call the {\em random edge process} :
%, and denote by $\RANDOMEDGEPROCESS$.
at time $t$, an edge
is given as an ordered pair of vertices $e_t = (v_t, w_t)$ chosen
uniformly at random, with repetition, from the set of all possible
$n^2$ ordered pairs (note that this model allows loops and repeated
edges). In what follows, we use $G_t$ to denote the graph induced by
the first $t$ edges, and given the orientation of each edge, use
$D_t$ to denote the directed graph induced by the first $t$ edges.
By $m_*$ we denote the time $t$ when the last vertex of degree one
in $G_t$ becomes a degree two vertex.

We will first prove that there exists an on-line algorithm $\ALG$
which $\whp$ orients the edges of the graph $G_{m_*}$ so that
the directed graph $D_{m_*}$ is Hamiltonian, and then in
Section \ref{sec:MainThmMod} show how Theorem \ref{thm_mainthm2} can
be recovered from this result.

\subsection{Organization of the Paper}
In the next section we describe the algorithm $\ALG$ that is used to
prove Theorem \ref{thm_mainthm2} (in the modified model). Then in
Section \ref{section_outline} we outline the proof of Theorem
\ref{thm_mainthm2}. Section \ref{sec:TypicalProperties} describes
several properties that a typical random edge process possesses.
Using these properties we prove Theorem \ref{thm_mainthm2} in
Section \ref{section_findingham}. Then in Section
\ref{sec:MainThmMod}, we show how to modify the algorithm $\ALG$, in
order to make it work for the original random graph process.

\medskip

\noindent \textbf{Notation.} A \emph{directed 1-factor} is a directed graph
in which every vertex has in-degree and out-degree exactly 1, and a
\emph{1-factor} of a directed graph is a spanning subgraph which is a directed
1-factor. The function $\exp(x) := e^{x}$ is the exponential
function. Throughout the paper $\log(\cdot)$ denotes the natural
logarithm. For the sake of clarity, we often omit floor and ceiling
signs whenever these are not crucial and make no attempts to
optimize our absolute constants. We also assume that the order $n$
of all graphs tends to infinity and therefore is sufficiently large
whenever necessary.

\section{The Orientation Rule}
\label{section_algorithm}

In this section we describe the algorithm $\ALG$. Its input is the
edge process ${\bf e} = (e_1, e_2, \ldots, e_{m_*})$, and output is
an on-line orientation of each edge $e_t$. The algorithm proceeds in two steps.
In the first step, which consists
of the first $2 n \log \log n$ edges, the algorithm builds a ``core"
which contains almost all the vertices, and whose edges are distributed (almost)
like a 6-in 6-out random graph. In the second step, which contains
all edges that follow, the remaining $o(n)$ non-core vertices are
taken care of, by being connected to the core in a way that
will guarantee $\whp$ the existence of a directed Hamiltonian cycle.

\subsection{Step I} \label{subsection_orientationrule1}
Recall that each edge is given as an ordered pair $(v,w)$. For every
vertex $v$ we keep a count of the number of times that $v$ appears
as the first vertex. We update the set of \emph{saturated} vertices,
which consists of the vertices which appeared at least 12 times as
the first vertex. Given the edge $(v,w)$ at time $t$, if $v$ is
still not saturated, direct the edge $(v,w)$ alternatingly with
respect to $v$ starting from an out edge (by alternatingly we mean,
if the last edge having $v$ as the first vertex was directed as an
out edge of $v$, then direct the current one as an in edge of $v$,
and vice-versa. For the first edge we choose arbitrarily the out
direction). Otherwise, if $v$ is saturated, then count the number of
times that $w$ appeared as a second vertex when the first vertex is
already saturated, and direct the edges alternatingly according to
this count with respect to $w$ starting from an in edge. This
alternation process is independent to the previous one. That is, even
if $w$ appeared as a first vertex somewhere before, the count should
be kept track separately from it.

For a vertex $v \in V$, let the \emph{first vertex degree} of $v$ be
the number of times that $v$ appeared as a first vertex in Step I,
and denote it as $\VTXONEDEG{v}{}$. Let the \emph{second vertex
degree} of $v$ be the number of times that $v$ appeared in Step I
as a second vertex of an edge whose first vertex is already
saturated, and denote it as $\VTXFREETWODEG{v}{}$. Note that the sum
of the first vertex degree and second vertex degree of $v$ is not
necessarily equal to the degree of $v$ in Step I as $v$ might
appear as a second vertex of an edge whose first vertex is not yet
saturated. We will call such an edge a \emph{neglected edge} of $v$.

\subsection{Step II} \label{subsection_orientationrule2}
Let $A$ be the set of saturated vertices at the end of Step I, and
$B=V\setminus A$. Call an edge an $A$-$B$ edge if one end point lies
in $A$ and the other end point lies in $B$, and similarly define
$A$-$A$ edges and $B$-$B$ edges. Given an edge $e=(v,w)$ at time
$t$, if $e$ is an $A$-$B$ edge, and w.l.o.g. assume that $v \in B$ and
$w \in A$, then direct $e$ alternatingly with respect to $v$, where
the alternation process of Step II continues the one from Step I
as follows:
\begin{enumerate}
  \setlength{\itemsep}{1pt} \setlength{\parskip}{0pt}
  \setlength{\parsep}{0pt}
\item If $v$ appeared as a first vertex in Step I at least once, then pick up where the alternation
process of $v$ as a first vertex in Step I stopped and continue the alternation.
\item If $v$ did not appear as a first vertex in Step I but did appear as a second vertex of an already saturated vertex,
then pick up where the alternation process of $v$ as a second vertex of a saturated vertex stopped in Step I and continue the alternation.
\item If $v$ appeared in Step I but does not belong to the above two cases, then consider the first neglected edge connected to $v$,
and start the alternation process from the opposite direction of this edge.
\item If none of the above, then start from an out edge.
\end{enumerate}
 Otherwise, if $e$ is an
$A$-$A$ edge or a $B$-$B$ edge, orient it uniformly at random. Note
that unlike Step~I, the order of vertices of the given edge
does not affect the orientation of the edge in Step II.

For a vertex $v \in B$, let the \emph{$A$-$B$ degree} of $v$ be the
number of $A$-$B$ edges incident to $v$ in Step II, and denote it
as $\PHASETWODEG{v}{}$. For $v \in A$, let $\PHASETWODEG{v}{} = 0$.

%\begin{remark}
%The orientation rule we use might not seem the most natural at first glance, and indeed it might be that a simpler orientation rule will work
%as well. However, we choose this rule for the sake of an easier analysis.
%\end{remark}

\section{Proof Outline}
\label{section_outline}

Our approach builds on Frieze's proof of the Hamiltonicity of the random
directed graph process \cite{Frieze} with some additional ideas.
His proof consists of two phases (the original proof consists of
three phases, but for simplicity, we describe it as two phases). We shall first
describe these two phases of Frieze's proof, and then point out the modifications
that are necessary to accommodate our different setting. Let $m = (1+o(1))n \log n $ be the
time at which the random directed graph process has
minimum in-degree and out-degree 1, and let $D_{n,m}$ be the
directed graph at time $m$
(throughout this section we say that random directed graphs
have certain properties if they have the properties $\whp$).

\subsection{Phase 1 : Find a small 1-factor}

In Phase 1, a 1-factor of $D_{n,m}$ consisting of at most $O(\log
n)$ cycles is constructed. To this end, a subgraph $\FIVEINOUT$ of
$D_{n,m}$ is constructed which uses only a small number of the
edges. Roughly speaking, for each vertex, use its first 5
out-neighbors and 5 in-neighbors (if possible) to construct $\FIVEINOUT$. Note
that the resulting graph will be similar to a random 5-in 5-out
directed graph, but still different as some vertices will only have
1 in-neighbor and 1 out-neighbor even at time $m$. Finally,
viewing $\FIVEINOUT$ as a bipartite graph $G'(V\cup V^*,E')$, where
$V^*$ is a copy of $V$, and $\{u,v^*\}\in E'$ iff $u \to v$ belongs to
$\FIVEINOUT$, one proves that $G'$ has a perfect matching. It turns
out that this matching can be viewed as a uniform random permutation of the
set of vertices $V$. A well known fact
about such permutations is that they $\whp$ consist of at most
$O(\log n)$ cycles.

\subsection{Phase 2 : Combining the cycles into a Hamilton cycle}

In Phase 2, the cycles of the 1-factor are combined into a
Hamilton cycle. The technical issue to overcome in this step is the
fact that in order to construct $\FIVEINOUT$, all of the edges were
scanned, and now supposedly we have no remaining random edges in the
process to combine the cycles of the 1-factor. However, note that
since $\FIVEINOUT$ consists of at most $10n$ edges, the majority of
edges need not be exposed. More rigorously, let $LARGE$ be the
vertices whose degree is $\Omega(\log n / \log \log n)$ at time
$t_0=2n \log n /3$ in the directed graph process. For the $LARGE$
vertices, its 5 neighbors in $\FIVEINOUT$ will be determined solely
by the edges up to time $t_0$, leaving the remaining edges (edges
after time $t_0$) of the process unexposed. Two key properties
used in Phase 2 are that $\whp$, $(a)$ $|LARGE|=n-o(n^{1/2})$,
and $(b)$ every cycle of the 1-factor contains many $LARGE$
vertices. Note that by $(a)$, out of the remaining $n\log n /3$ edges,
all but $o(1)$-fraction will connect two $LARGE$ vertices. Phase 2
can now be summarized by the following theorem \cite{Frieze}.

\begin{THM} \label{thm_friezephase23}
Let $V$ be a set of $n$ vertices and $L \subset V$ be a subset of
size at least $n - o(n^{1/2})$. Assume that $D$ is a directed
1-factor over $V$ consisting of at most $O(\log n)$ cycles, and the
vertices $V \setminus L$ are at distance at least 10 away from each
other in this graph.

If $(1/3-o(1))n\log n$ $L$-$L$ edges are given uniformly at random,
then $\whp$ the union of these edges and the graph $D$ contains a
directed Hamilton cycle.
\end{THM}

The proof of a slightly stronger version of Theorem \ref{thm_friezephase23}
will be given in Section \ref{sec:MainThmMod}.

\subsection{Comparing with our setting}
The main technical issue in this paper is to reprove Phase 1, namely,
the existence of a 1-factor with small number of cycles.
In \cite{Frieze}, the fact that
all vertices have the same distribution in $\FIVEINOUT$, led to an argument
showing the existence of a matching that translates into a uniform
random permutation. Our case is different because of the orientation
rule. We have different types of vertices each being oriented in a
different way, breaking the nice symmetry. The bulk of our technical
work is spent in resolving this technical issue.

Once this is done, that is after achieving the 1-factor, we
come up with an analogue of $LARGE$, which we call ``saturated".
Similarly as in Phase 2 described above, we prove that $\whp$
$(a')$ most of the vertices are saturated, and $(b')$ every cycle in
the 1-factor contains many saturated vertices. However, the naive
approach results in a situation where one cannot apply Theorem
\ref{thm_friezephase23} ($(a')$ and $(b')$ are quantitatively weaker
than $(a)$ and $(b)$). Thus we develop the argument of
``compressing'' vertices of a given cycle. This idea allows us to
get rid of all the non-saturated vertices, leading to another graph
which only has saturated vertices in it. Details will be given in
Section \ref{subsection_small1factor}. Once we apply the compression
argument, we can use Theorem \ref{thm_friezephase23} to finish the
proof. Let us mention that the compression argument can be applied
after Phase 1 in \cite{Frieze} as well to simplify the proof.

\section{A Typical Random Process}\label{sec:TypicalProperties}

The following well-known concentration result (see, for example
\cite[Corollary A.1.14]{AlSp}) will be used several times in the
proof. We denote by $Bi(n,p)$ the binomial random variable with
parameters $n$ and $p$.
\begin{THM}(Chernoff's inequality) If $X \sim Bi(n,p)$ and $\varepsilon>0$, then
\[ \BFP\big(|X - \mathbb{E}[X]| \geq \varepsilon \mathbb{E}[X]\big) \leq e^{-\Omega_\varepsilon(\mathbb{E}[X])}. \]
\end{THM}

\subsection{Classifying Vertices}

To analyze the algorithm it will be convenient to work with three
sets of vertices. The first is the set of \textbf{saturated}
vertices at Step I. Throughout we will use $A$ to denote this set.
Let us now consider the non-saturated vertices $B=V\setminus A$.
Here we distinguish between two types. We say that $v \in B$
\textbf{blossoms} if there are at least 12 edges of the form
$\{v,A\}$ in Step II (by $A$ we mean an arbitrary vertex from $A$),
and let $B_1$ be the collection of vertices which blossom.  All the
remaining vertices are \textbf{restricted}, and is denoted by $B_2$.
Thus every vertex either is saturated ($A$), blossoms ($B_1$), or
is restricted ($B_2$).

Furthermore, the set of restricted vertices has two important
subclasses which are determined by the first vertex
degree $\VTXONEDEG{v}{}$, second vertex
degree $\VTXFREETWODEG{v}{}$, and $A$-$B$
degree $\PHASETWODEG{v}{}$ defined in the previous section. We say
that a restricted vertex $v$ \textbf{partially-blossoms} if the sum
of its first vertex degree, second vertex degree, and $A$-$B$ degree
is at least 2. Note that since we stopped the process when the graph
has minimum degree 2, every vertex $v$ has degree at least 2. Thus,
if the above mentioned sum is at most 1, then $v$ either has a
neglected edge, or a $B$-$B$ edge connected to it. A useful fact
that we prove in Lemma \ref{claim_exhastiveness} says that $\whp$
all such vertices $v$ have one $A$-$B$ edge (thus $\PHASETWODEG{v}{}
= 1$), and at least one neglected edge. Thus, we call a restricted
vertex $v$ not being partially-blossomed, and having one $A$-$B$
edge and at least one neglected edge as a \textbf{bud}.

\subsection{Properties of a Typical Random Process}

In this section we list several properties that hold $\whp$ for random edge processes. We will call
an edge process \textbf{typical} if indeed the properties hold. Let
$$m_1 = \frac{1}{2} n \log n + \frac{1}{2} n \log \log
n - n \log \log \log n, \qquad m_2 = \frac{1}{2} n \log n +
\frac{1}{2} n \log \log n + n \log \log \log n.$$ Note that for
a fixed vertex $v$, the probability of an edge being incident to $v$
is $\frac{2n-1}{n^2} = \frac{2}{n} - \frac{1}{n^2}$ (this is because
in our process, each edge is given by an ordered pair of vertices). However as it
turns out the small order term $\frac{1}{n^2}$ is always negligible for
our purpose, so we will use the probability $\frac{2}{n}$ for this
event, and remind the reader that the term
$\frac{1}{n^2}$ is omitted. Recall that the stopping time $m_*$ is the time
at which the last vertex of degree one becomes a degree two vertex and the
process stops.

\begin{CLAIM} \label{clm_stopping_time}
Let $m_*$ be the stopping time of the random process. Then $\whp$
$$m_1 \le m_*\le m_2.$$
\end{CLAIM}
\begin{proof}
For a fixed vertex $v$, the probability of an edge being incident to
$v$ is about $\frac{2}{n}$. Hence the probability of $v$ having
degree at most 1 at time $m_2$ is,
\[ \left(1 - \frac{2}{n}\right)^{m_2} + {m_2 \choose 1} \frac{2}{n} \cdot \left(1 - \frac{2}{n}\right)^{m_2 - 1}
\le  3 \log n \cdot  e^{- \log n - \log \log n - 2 \log \log \log n
} = O\left(\frac{1}{n (\log \log n)^2}\right). \] Thus by Markov's
inequality, $\whp$ there is no vertex of degree at most 1 after
$m_2$ edges. This shows that $m_* \le m_2$. Similarly, the expected
number of vertices having degree at most 1 after seeing $m_1$
edges is $\Omega((\log \log n)^2)$, and by computing the second
moment of the number of vertices having degree at most 1, we can
show that after $m_1$ edges $\whp$ at least one such vertex exits.
This shows that $m_* \ge m_1$. The rest of the details are fairly
standard and are omitted.
\end{proof}

Next we are going to list some properties regarding the different
types of vertices.
\begin{CLAIM} \label{clm_sizes_of_A_B}
The number of saturated vertices satisfies $\whp$
$$|A| \ge n\left(1- \frac{(\log
\log n)^{12}}{\log^2 n}\right).$$
\end{CLAIM}
\begin{proof}
For a fixed vertex $v$, the probability of $v$ occurring as the
first vertex of an edge is (exactly) $\frac{1}{n}$, and thus the
probability of $v$ ending up non-saturated at Step I is at most
\begin{align*}
\sum_{k=0}^{11} {2 n \log \log n \choose k} \left( \frac{1}{n} \right)^k \cdot \left(1 - \frac{1}{n}\right)^{2 n \log \log n - k}
 \le \sum_{k=0}^{11} (2 \log \log n)^k \frac{1}{\log ^2 n} = O\left(\frac{(\log \log n)^{11}}{\log ^2 n}\right).
\end{align*}
The claim follows from Markov's inequality.
\end{proof}
Our next goal is to prove that the restricted vertices consist only of partially-blossomed and bud vertices.
For that we need the following auxiliary lemma.

\begin{CLAIM} \label{clm_bbedges}
Let $E_{BB}$ be the collection of all $B$-$B$ edges (in Step II). The graph $G_{m_*}\setminus E_{BB}$ has $\whp$ minimum degree 2.
\end{CLAIM}
\begin{proof}
If the graph $G_{m_*}\setminus E_{BB}$ has minimum degree less than
2 for some edge process ${\bf e}$, then there exists a vertex $v$
which gets at most one edge other than a $B$-$B$ edge, and at least
one $B$-$B$ edge. By Claim \ref{clm_stopping_time}, it suffices to
prove that the graph $\whp$ does not contain a vertex which has at
most one edge other than a $B$-$B$ edge at time $m_1$, and at least
one $B$-$B$
edge at time $m_2$. %(we are using different times for the two events
%because one the probability of the former is decreasing, and the
%later is increasing as time passes).
Let $\mathcal{A}_v$ be the
event that $v$ is such vertex. Let $\mathcal{BS}$ be the event that
$|B|\le \frac{(\log \log n)^{12}}{\log ^2 n} n $ ($B$ is small),
and note that $\BFP(\mathcal{BS}) = 1 - o(1)$ by Claim \ref{clm_sizes_of_A_B}.
%Since $P(\mathcal{BS}) = 1 - o(1)$ by Claim \ref{clm_sizes_of_A_B}, we will
%condition on the event that $\mathcal{BS}$ holds without further mentioning it.
%Then we have
Then we have
\begin{align} \label{eq:eq1}
\BFP(\text{$G_{m_*}\setminus E_{BB}$ has minimum degree less than 2}) =
\BFP\left( \bigcup_{v\in V}{\mathcal{A}_v} \right) \le n \cdot
\BFP\left(\mathcal{A}_v \cap \mathcal{BS}\right) + o(1).
\end{align}
%In the last inequality we used Claim \ref{clm_sizes_of_A_B} which guarantees that $\whp$ $|B| \le
%\frac{(\log \log n)^{12}}{\log ^2 n} n$.
The event $\mathcal{A}_v$ is equivalent to the vertex $v$ receiving
$k$ $B$-$B$ edges, for some $k > 0$, and at most one edge other than
a $B$-$B$ edge at appropriate times. This event is contained in the
event $\mathcal{C}_v \cap \mathcal{D}_{v,k}$ where $\mathcal{C}_v$
is the event ``$v$ appears at most once in Step I", and
$\mathcal{D}_{v,k}$ is the event ``$\PHASETWODEG{v}{} \le 1$ by time
$m_1$ and $v$ receives $k$ $B$-$B$ edges by time $m_2$". Therefore
our next goal is to bound
\begin{align}\label{eqn:eq2}
\BFP({\mathcal{C}_v} \cap {\mathcal{D}_{v,k}} \cap \mathcal{BS}) =
\BFP({\mathcal{C}_v} \cap \mathcal{BS})\cdot
\BFP({\mathcal{D}_{v,k}}|{\mathcal{C}_v \cap \mathcal{BS}})
\le \BFP({\mathcal{C}_v})\cdot
\BFP({\mathcal{D}_{v,k}}|{\mathcal{C}_v \cap \mathcal{BS}}).
\end{align}
We can bound the probability of the event $\mathcal{C}_v$ by,
\begin{align}
\label{eqn:eq3}
\left(1 - \frac{2}{n}\right)^{2 n \log \log n} + {2 n \log \log n \choose 1} \left( \frac{2}{n} \right) \cdot \left(1 - \frac{2}{n}\right)^{2 n \log \log n - 1}
 = O\left(\frac{\log \log n}{\log ^4 n}\right).
\end{align}
To bound the event $\mathcal{D}_{v,k}$ which is ``$\PHASETWODEG{v}{}
\le 1$ at time $m_1$ and $v$ receives $k$ $B$-$B$ edges by time
$m_2$", note that $\mathcal{C}_v$ and $\mathcal{BS}$
are events which depend only on
the first $2\log \log n$ edges (Step I edges). Therefore
conditioning on this event does not affect the distribution of edges
in Step II (each edge is chosen uniformly at random among all
possible $n^2$ pairs). We only consider the case
$\PHASETWODEG{v}{}=1$ (the case $\PHASETWODEG{v}{}=0$ can be
handled similarly, and turns out to be dominated by the
case $\PHASETWODEG{v}{}=1$).
Thus to bound the probability, we choose
$k+1$ edges among the $m_2 - 2n\log \log n$ edges, let 1 of them to
be an $A$-$B$ edge, $k$ of them to be $B$-$B$ edges incident to $v$.
Moreover, since $d_{AB}(v) \le 1$ at time $m_1$, we know that at
least $m_1 - 2n\log \log n - k -1$ edges are not incident to $v$.
Thus,
\begin{align*}
\BFP &({\mathcal{D}_{v,k}} \,|\, {\mathcal{C}_v \cap \mathcal{BS}}) \\
& \le {m_2 - 2 n \log \log n \choose k+1} \left(\frac{2}{n}\right)^{k+1} {k+1 \choose 1} \frac{|A|}{n} \left(\frac{|B|}{n}\right)^k \left(1 - \frac{2}{n}\right)^{m_1 - 2n \log \log n - k - 1}.
\end{align*}
By using the inequalities $1 - x \le
e^{-x}$, $|A| \le n$, and ${m_2 - 2 n \log \log n \choose k+1} \le
m_2^{k+1}$, the probability above is bounded by
\begin{align}\label{eqn:eq4}
(k+1)m_2^{k+1} \left(\frac{2}{n}\right)^{k+1} \left(\frac{|B|}{n}\right)^k \exp \left( -\frac{2}{n} (m_1 - 2n \log \log n - k - 1)\right).
\end{align}
Therefore by (\ref{eqn:eq2}), (\ref{eqn:eq3}), and (\ref{eqn:eq4}),
\begin{align*}
\BFP &({\mathcal{C}_v} \cap \mathcal{D}_{v,k} \cap \mathcal{BS}) \leq \\
&O\left(\frac{\log\log n }{\log^4 n}\right) (k+1) m_2^{k+1} \left(\frac{2}{n}\right)^{k+1} \left(\frac{|B|}{n}\right)^k \exp \left( -\frac{2}{n} (m_1 - 2n \log \log n - k - 1)\right).
\end{align*}
Plugging the bound $|B| \le \frac{n \log \log ^{12}n}{\log^2 n}$ and
$m_2 \le n \log n$ in the latter, one obtains:
\[ O(k)\left(\frac{\log\log n }{\log^3 n}\right) \left(\frac{2(\log \log n)^{12}}{\log n}\right)^{k} \exp \left( -\frac{2}{n} (m_1 - 2 n \log \log n - k - 1) \right). \]
By the definition $m_1 = \frac{1}{2}n \log n + \frac{1}{2}n \log
\log n - n \log \log \log n$, this further simplifies to
\[ O(k)\left( \frac{(\log \log n)^3}{n} \right) \left(\frac{2 e^{2/n}(\log \log n)^{12}}{\log n}\right)^{k}. \]
Summing over all possible values of $k$,
\begin{align*}
\sum_{k=1}^{\infty} \BFP ({\mathcal{C}_v} \cap \mathcal{D}_{v,k} \cap \mathcal{BS})\le  \sum_{k=1}^{\infty} \frac{O(k) (\log \log n)^3}{n}\left( \frac{4(\log \log n)^{12}}{\log n}\right)^{k} = o(n^{-1}).
\end{align*}
Going back to $(\ref{eq:eq1})$, we get that
$$\BFP(\text{$G_{m_*}\setminus E_{BB}$ has minimum degree less than 2})=n\cdot o(n^{-1})+o(1)=o(1).$$
Note that as mentioned in the beginning of this section, we used $\frac{2}{n}$ to
estimate the probability of an edge being incident to a fixed vertex. This probability
is in fact $\frac{2}{n} - \frac{1}{n^2}$, but the term $\frac{1}{n^2}$ will only
affect the lower order estimates.
\end{proof}

\begin{CLAIM} \label{claim_exhastiveness}
Every restricted vertex is $\whp$ either partially-blossomed, or a bud.
\end{CLAIM}
\begin{proof}
Assume there exists a restricted vertex $v$ which is not
partially-blossomed or a bud. Then by definition, the sum
$\VTXONEDEG{v}{} + \VTXFREETWODEG{v}{} + \PHASETWODEG{v}{} \le 1$.
The possible values of the degrees
$(\VTXONEDEG{v}{},\VTXFREETWODEG{v}{},\PHASETWODEG{v}{})$ are
$(1,0,0), (0,1,0), (0,0,1)$, or $(0,0,0)$. Vertices which correspond
to $(0,0,1)$ will all be bud vertices $\whp$ by Claim \ref{clm_bbedges}. It
suffices to show then that $\whp$ there does not exist vertices
which correspond to $(1,0,0), (0,1,0)$, or $(0,0,0)$. Let $T$ be the
collection of vertices which have $\VTXONEDEG{v}{} +
\VTXFREETWODEG{v}{} \le 1$ and $\PHASETWODEG{v}{} = 0$ at time
$m_1$. By Claim \ref{clm_stopping_time} it suffices to prove that
$T$ is empty.
%Again, let $\mathcal{BS}$ be the event $|B| \le
%\frac{(\log\log n)^{12}}{\log^2n}n$, and note that by Claim
%\ref{clm_sizes_of_A_B}, $\BFP(\mathcal{BS}) = 1 - o(1)$. Thus
%without further mentioning, we will condition on the event that
%$\mathcal{BS}$ holds throughout the proof.
Let $\mathcal{BS}$ be the event $|B| \le
\frac{(\log\log n)^{12}}{\log^2n}n$, and note that by Claim
\ref{clm_sizes_of_A_B}, $\BFP(\mathcal{BS}) = 1 - o(1)$.
The event $\{ T \neq \emptyset \}$ is the same as $\cup_{v \in V} \{ v \in T \}$, and thus by the union bound,
\begin{align*}
\BFP(T \neq \emptyset) &\le o(1) + \sum_{v \in V} \BFP(\{v \in T\} \cap \mathcal{BS}) \\
&= o(1) + \sum_{v \in V}
\BFP\left(\{\VTXONEDEG{v}{}+\VTXFREETWODEG{v}{} \le 1\} \cap \{\PHASETWODEG{v}{} = 0\} \cap \mathcal{BS} \right). \end{align*}
By Bayes equation, the second term of right hand side splits into,
\begin{align}
&&\sum_{v \in V} \BFP\left(\{\VTXONEDEG{v}{}+\VTXFREETWODEG{v}{} \le 1\} \cap \mathcal{BS} \right) \cdot \BFP\left(\PHASETWODEG{v}{} =
0 \,|\, \{\VTXONEDEG{v}{}+\VTXFREETWODEG{v}{} \le 1\} \cap \mathcal{BS} \right) \nonumber \\
&&\le \sum_{v \in V} \BFP\left(\VTXONEDEG{v}{}+\VTXFREETWODEG{v}{} \le 1 \right) \cdot \BFP\left(\PHASETWODEG{v}{} =
0 \,|\, \{\VTXONEDEG{v}{}+\VTXFREETWODEG{v}{} \le 1\} \cap \mathcal{BS} \right). \label{eqn1}
\end{align}
The probability $\BFP(\VTXONEDEG{v}{}+\VTXFREETWODEG{v}{} \le 1)$ can be bounded by $\BFP(\{\VTXONEDEG{v}{} \le 1\}\cap\{\VTXFREETWODEG{v}{} \le 1\})$ which satisfies,
$$\BFP(\{\VTXONEDEG{v}{} \le 1\}\cap\{\VTXFREETWODEG{v}{} \le 1\})
= \BFP(\VTXONEDEG{v}{} \le 1) \cdot \BFP(\VTXFREETWODEG{v}{} \le 1 \, | \, \VTXONEDEG{v}{} \le 1). $$
The term $\BFP(\VTXONEDEG{v}{} \le 1)$ can be easily calculated as,
\begin{align*}
\left(1 - \frac{1}{n}\right)^{2 n \log \log n} + {2 n \log \log n \choose 1} \left( \frac{1}{n} \right) \cdot \left(1 - \frac{1}{n}\right)^{2 n \log \log n - 1}
 = O\left(\frac{\log \log n}{\log ^2 n}\right).
\end{align*}
To estimate $\BFP(\VTXFREETWODEG{v}{} \le 1 \, | \, \VTXONEDEG{v}{} \le 1)$, expose
the edges of Step I as follows: First expose all the first vertices.
Then expose the second vertices whose first vertex is saturated
($\VTXFREETWODEG{v}{}$ is now determined for every $v \in V$). The number of second-vertex-spots
that are considered is at least $2n \log \log n - 12 n$, and thus $\BFP(\VTXFREETWODEG{v}{} \le 1 | \VTXONEDEG{v}{} \le 1)$ is at most
\begin{align*}
\left(1 - \frac{1}{n}\right)^{2 n \log \log n - 12 n} + {2 n \log \log n \choose 1} \left( \frac{1}{n} \right) \cdot \left(1 - \frac{1}{n}\right)^{2 n \log \log n - 12 n - 1}
 = O\left(\frac{\log \log n}{\log ^2 n}\right).
\end{align*}
Thus as a crude bound, we have
\[ \BFP(d_1(v) + d_2(v) \le 1) \le \BFP(\VTXONEDEG{v}{} \le 1) \cdot \BFP(\VTXFREETWODEG{v}{} \le 1 \, | \, \VTXONEDEG{v}{} \le 1) =
O\left( \frac{(\log \log n)^2}{\log^4 n}\right). \]
Since $\VTXONEDEG{v}{}+ \VTXFREETWODEG{v}{} \le 1$ implies that $v
\in B$, and $\PHASETWODEG{v}{}$ depends only on the Step II edges
(which are independent from $\VTXONEDEG{v}{}$, $\VTXFREETWODEG{v}{}$, and $\mathcal{BS}$),
the second term of the right hand
side of equation (\ref{eqn1}), the probability
$\BFP\left(\PHASETWODEG{v}{} = 0 \,|\, \{\VTXONEDEG{v}{}+
\VTXFREETWODEG{v}{} \le 1\} \cap \mathcal{BS} \right)$ can be
bounded by
\begin{align*}
& \left( 1 - 2\frac{1}{n}\frac{|A|}{n} \right)^{m_1 - 2 n \log \log n} \le \exp\left(- 2(m_1 - 2 n \log \log n)|A|/n^2\right) \\
\le& \exp \left(- (\log n - 3 \log \log n - 2 \log \log \log n)\left(1 - \frac{(\log \log n)^{12}}{\log^2 n} \right)\right) \\
\le& \exp \left(- \log n + 3 \log \log n + 2 \log \log \log n + o(1) \right) = O\left(\frac{(\log n)^3 (\log \log n)^2}{n}\right).
\end{align*}
Therefore in (\ref{eqn1}),
\begin{align*}
\BFP(T \neq \emptyset)
&\le o(1) + \sum_{v \in V}  O\left(  \frac{(\log \log
n)^{2}}{\log ^4 n} \right) O\left(\frac{(\log n)^3 (\log \log n)^2}{n} \right) \\
&= o(1) + O\left( \frac{(\log \log n)^{4}}{\log n} \right) = o(1).
\end{align*}
\end{proof}

\begin{CLAIM}\label{clm_properties_of_buds}
The following properties hold $\whp$ for \emph{restricted} vertices:
\begin{enumerate}[(i)]
  \setlength{\itemsep}{1pt} \setlength{\parskip}{0pt}
  \setlength{\parsep}{0pt}
  \item There are at most $\log^{13}n$ such vertices,
  \item every such two vertices are at distance at least 3 in $G_{m_*}$ from each other.
\end{enumerate}
\end{CLAIM}
\begin{proof}
Since being a restricted vertex is a monotone decreasing property,
by Claim \ref{clm_stopping_time} it suffices to prove $(i)$ at time
$m_1$. Recall that $B_2$ is the collection of restricted vertices (a
vertex is restricted if it is not saturated or blossomed).

First, condition on the whole outcome of Step I edges (first
$2n\log \log n$ edges) and the event that $|B| \le \frac{(\log \log n)^{12} }{\log^2
n} n$. Then the set $B$ is determined, and
for a vertex $v \in B$, we can bound the probability of the event
$v \in B_2$ as following
\begin{align}
\BFP\left( v \in B_2 \right) \le \sum_{\ell=0}^{11} {m_2 \choose \ell}\left(\frac{2}{n}\right)^\ell
\left( 1 - \frac{2|A|}{n^2} \right)^{m_1 - 2 n \log \log n - \ell} \label{eq:eq5}.
\end{align}
Use the inequalities $m_1 = \frac{1}{2} n \log n + \frac{1}{2} n
\log \log n - \log \log \log n \le n \log n$, $m_2 \le n \log n$, $1
- x \le e^{-x}$, and $|A| = n - |B| \ge n \left(1 - \frac{(\log \log n)^{12}}{\log^2
n}\right)$ to bound the above by
\[
\sum_{l=0}^{11} (2\log n)^\ell
\exp\left(- (\log n - 3\log \log n - 2 \log \log \log n - \ell)\left(1 - \frac{(\log \log n)^{12}}{\log^2 n} \right) \right).
\]
The sum is dominated by $\ell=11$, and this gives
\[
O\left( \log^{11} n \right)
  \exp\left(- \log n + 3 \log \log n + 2 \log \log \log n + o(1) \right) \le O\left( \frac{(\log \log n)^{2}\log^{14} n}{n} \right).
\]
Thus the expected size of $B_2$ given the Step I edges is
\[ \BBE[|B_2| \,|\, \textrm{Step I edges}] \leq |B| \cdot O\left( \frac{(\log \log n)^{2}\log^{14} n}{n} \right) \le O((\log \log n)^{14} \log^{12}n). \]
Since the assumptions on $A$ and $B$ holds $\whp$ by Claim \ref{clm_sizes_of_A_B},
we can use Markov inequality to conclude that $\whp$ there
are at most $\log^{13}n$ vertices in $B_2$. Let us now prove $(ii)$.

For three distinct vertices $v_1, v_2$ and $w$ in $V$, let
$\mathcal{A}(v_1, v_2, w)$ be the event that $w$ is a common
neighbor of $v_1$ and $v_2$. The probability of there being edges
$(v_1,w)$ (or $(w,v_1)$) and $(v_2, w)$ (or $(w,v_2)$) and $v_1, v_2
\in B_2$ can be bounded by first choosing two
time slots where $(v_1,w)$ (or $(w,v_1)$) and $(v_2, w)$ (or $(w,v_2)$)
will be placed, and then filling in the remaining edges so that
$v_1, v_2 \in B_2$. We will only bound the event of there being edges
$(v_1, w)$ and $(w, v_2)$ in the edge process (other cases can
be handled in a similar manner). The probability we would like to bound is
\[ \BFP(\exists v_1, v_2, w, \exists 1 \le t_1, t_2 \le m_2,  e_{t_1} = (v_1, w), e_{t_2} = (w, v_2), v_1, v_2 \in B_2). \]
By the union bound this probability is at most
\begin{align}
&\sum_{v_1, v_2,w \in V} \sum_{t_1, t_2=1}^{m_2} \BFP(e_{t_1} = (v_1, w), e_{t_2} = (w, v_2), v_1, v_2 \in B_2) \\
= & \sum_{v_1, v_2,w \in V} \sum_{t_1, t_2=1}^{m_2} \BFP(e_{t_1} = (v_1, w), e_{t_2} = (w, v_2)) \BFP(v_1, v_2 \in B_2 | e_{t_1}=(v_1,w), e_{t_2} = (w,v_2)) \nonumber \\
\le & \frac{1}{n} \sum_{t_1, t_2=1}^{m_2} \BFP(v_1, v_2 \in B_2 | e_{t_1}=(v_1,w), e_{t_2} = (w,v_2)) \label{eq:eq6}.
\end{align}
To simplify the notation we abbreviate $\BFP(v_1, v_2 \in B_2 | e_{t_1}=(v_1,w), e_{t_2} = (w,v_2))$ by $\BFP(v_1, v_2 \in B_2 |e_{t_1},e_{t_2})$. By using the independence of Step I and Step II edges we have,
\[ \BFP(v_1, v_2 \in B_2 |e_{t_1},e_{t_2}) = \BFP(v_1, v_2 \in B |e_{t_1},e_{t_2}) \BFP(v_1, v_2 \notin B_1 |v_1,v_2 \in B, e_{t_1}, e_{t_2}). \]
For fixed $t_1$ and $t_2$, we can bound $\BFP(v_1, v_2 \in B |e_{t_1},e_{t_2})$ by the probability
of ``$v_1$ and $v_2$ appear at most 22 times combined in Step I as a first vertex other than at time $t_1$ and $t_2$'',
whose probability can be bounded as follows regardless of the value of $t_1$ and $t_2$,
\begin{align*}
\sum_{k=0}^{22} {2 n \log \log n \choose k} \left( \frac{2}{n} \right)^k \cdot \left(1 - \frac{2}{n}\right)^{2 n \log \log n - 2 - k}
 \le \sum_{k=0}^{22} (4 \log \log n)^k \frac{O(1)}{\log ^4 n} = O\left(\frac{(\log \log n)^{22}}{\log ^4 n}\right).
\end{align*}
To bound $\BFP(v_1, v_2 \notin B_1 |v_1,v_2 \in B, e_{t_1}, e_{t_2})$, it suffices to bound $\BFP(v_1, v_2 \notin B_1 |v_1,v_2 \in B, e_{t_1}, e_{t_2}, \mathcal{BS})$,
which can be bounded
by the probability of ``$v_1$ and $v_2$ receives at most 22 $A$-$B$ edges combined in Step II
other than at time $t_1$ and $t_2$''.
Regardless of the value of $t_1$ and $t_2$, this satisfies the bound,
\[
\sum_{\ell=0}^{22} {m_2 \choose \ell}\left(\frac{4}{n}\right)^{\ell}
\left( 1 - \frac{4}{n}\frac{|A|}{n} \right)^{m_1 - 2 - 2 n \log \log n - l}.
\]
Note that $\frac{4}{n}$ and $\frac{2}{n}$ in this equation should in
fact involve some terms of order $\frac{1}{n^2}$, but we omitted it
for simplicity since it does not affect the asymptotic final
outcome. By a similar calculation to (\ref{eq:eq5}), this eventually
can be bounded by $O(\frac{\log^{29}n}{n^2})$. Thus we have
\[ \BFP(v_1, v_2 \in B_2 |e_{t_1},e_{t_2}) = O\left(\frac{\log^{26}n}{n^2}\right), \]
which by (\ref{eq:eq6}) and $m_2 \le n \log n$ gives,
\[ \BFP(\exists v_1, v_2, w, \exists 1 \le t_1, t_2 \le m_2,  e_{t_1} = (v_1, w), e_{t_2} = (w, v_2), v_1, v_2 \in B_2) \le O\left(\frac{\log^{28}n}{n}\right). \]
Therefore by Markov's inequality, $\whp$ no such three vertices exist,
which implies that two vertices $v_1, v_2 \in B_2$ cannot be at
distance two from each other in $G_{m_*}$.
Similarly, we can prove that $\whp$ every two vertices $v_1, v_2 \in B_2$
are not adjacent to each other, and hence
$\whp$ every $v_1, v_2 \in B_2$ are at distance
at least two away from each other.
\end{proof}

\subsection{Configuration of the edge process}

To prove that our algorithm succeeds $\whp$, we first reveal some
pieces of information of the edge process, which we call the
``configuration'' of the process. These information will allow
us to determine whether the underlying edge process is typical or not.
Then in the next section, using
the remaining randomness,
we will construct a Hamilton cycle.

In the beginning, rather than thinking of edges coming one by one,
we regard our edge process
${\bf e} = (e_1, e_2, \cdots, e_{m_*})$ as a collection of edges $e_i$
for $i=1, \cdots, m_*$
whose both endpoints are not known. We can decide to
reveal certain information as necessary.
Let us first reveal the following.

\begin{enumerate}
\item[1.] For $t \leq 2 n \log \log n$, reveal the first vertex of
the $t$-th edge $e_t$. If this vertex already appeared as the first vertex
at least 12 times among the edges $e_1, \cdots, e_{t-1}$, then
also reveal the second vertex.
\end{enumerate}

Given this information, we can determine the saturated vertices, and
hence we know the sets $A$ and $B$. Therefore, it is possible to
reveal the following information.

\begin{enumerate}
\item[2.] For $t > 2 n \log \log n$, reveal all the vertices
that belong to $B$.
\end{enumerate}

The information we revealed
determines the blossomed ($B_1$), and restricted ($B_2$) vertices.
Thus we can further reveal the following information.

\begin{enumerate}
\item[3.] For $t \leq 2 n \log \log n$, further reveal all the non-revealed
vertices that belong to $B_2$.
\item[4.] For every edge $e_t=(v_t, w_t)$ in which we already know
that either $v_t \in B_2$ or $w_t \in B_2$, also reveal the
other vertex.
\end{enumerate}
We define the \emph{configuration} of an edge process as the above
four pieces of information.

We want to say that all the non-revealed vertices are uniformly
distributed over certain sets. But in order for this to be true, we
must make sure that the distribution of the non-revealed vertices
is not affected by the fact that we know the value of $m_*$ (some
vertex has degree exactly 2 at time $m_*$, and maybe a non-revealed
vertex will make this vertex to have degree 2 earlier than $m_*$).
This is indeed the case, since the last vertex to have degree 2 is
necessarily a restricted vertex, and all the locations of the
restricted vertices are revealed. Thus the non-revealed vertices
cannot change the value of $m_*$. Therefore, once we condition on
the configuration of an edge process, the remaining vertices are
distributed in the following way:

\begin{enumerate}
  \setlength{\itemsep}{1pt}
  \setlength{\parskip}{0pt}
  \setlength{\parsep}{0pt}
\item[(i)] For $t \leq 2 n \log \log n$, if the first vertex of the edge $e_t$
appeared at most 12 times among $e_1, \cdots, e_{t-1}$, then its second
vertex is either a known vertex in $B_2$ or is a random vertex in $V \setminus B_2$.
\item[(ii)] For $t > 2 n \log \log n$, if both vertices of $e_t$ are not revealed,
then $e_t$ consists of two random vertices of $A$. If only one of the
vertices of $e_t$ is not revealed, then the revealed vertex is in $B$,
and the non-revealed vertex is a random vertex of $A$.
\end{enumerate}

\begin{DEF} \label{def_typicalconfiguration}
A configuration of an edge process is \emph{typical} if it
satisfies the following.
\begin{enumerate}[(i)]
  \setlength{\itemsep}{1pt}
  \setlength{\parskip}{0pt}
  \setlength{\parsep}{0pt}
\item The number of saturated and blossomed vertices satisfy $|A| \ge n - \frac{(\log\log n)^{12}}{\log^2 n} n$, and $|B_1| \le \frac{(\log \log n)^{12}}{\log^2 n} n$ respectively.
\item The number of restricted vertices satisfies $|B_2| \le \log^{13} n$.
\item Every vertex appears at least twice in the configuration even without considering
the $B$-$B$ edges.
\item All the restricted vertices are either partially-blossomed or buds.
\item In the non-directed graph induced by the edges whose both endpoints are revealed,
every two restricted vertices $v_1, v_2$ are at distance
at least 3 away from each other.
\item There are at least $\frac{1}{3} n \log n$ edges $e_t$ for $t > 2n\log \log n$
whose both endpoints are not yet revealed.
\end{enumerate}
\end{DEF}

\begin{LEMMA} \label{lemma_typicalconfiguration}
The random edge process has a typical configuration $\whp$.
\end{LEMMA}
\begin{proof}
The fact that the random edge process has $\whp$ a configuration
satisfying $(i), (iii)$, and $(iv)$ follows from Claims
\ref{clm_sizes_of_A_B}, \ref{clm_bbedges}, and \ref{claim_exhastiveness} respectively.
$(ii)$ and $(v)$ follow from
Claim \ref{clm_properties_of_buds}. To verify $(vi)$, note that by Claim \ref{clm_stopping_time} and \ref{clm_sizes_of_A_B}, $\whp$~ there are at least
$\frac{1}{2} n \log n - 2 n \log \log n$ edges of Step II, and $|A| = (1
- o(1))n$. Therefore the probability of a Step II edge being an $A$-$A$
edge is $1 - o(1)$, and the expected number of $A$-$A$ edges is
$(1/2-o(1))n \log n$. Then by Chernoff's inequality, $\whp$ there are at
least $\frac{1}{3} n \log n$ $A$-$A$ edges. These edges are
the edges we are looking for in $(vi)$.
\end{proof}

\section{Finding a Hamilton Cycle}
\label{section_findingham}

In the previous section, we established several useful properties
of the underlying graph $G_{m_*}$.
In this section, we study the algorithm $\ALG$ using these properties,
and prove that conditioned on the edge process having a typical
configuration, the graph $D_{m_*}$ $\whp$ contains a Hamilton cycle (recall
that the graph  $D_{m_*}$  is the set of random edges of the edge
process, oriented according to \ALG). As described in Section \ref{section_outline}, the proof is a
constructive proof, in the sense that we describe how to find such a
cycle. The algorithm is similar to that used in \cite{Frieze} which we described in some details
in Section \ref{section_outline}. Let us briefly recall that it proceeds in two stages:
\begin{enumerate}
  \setlength{\itemsep}{1pt} \setlength{\parskip}{0pt}
  \setlength{\parsep}{0pt}
  \item Find a 1-factor of $G$. If it contains more than $O(\log n)$ cycles, fail.
  \item Join the cycles into a Hamilton cycle.
 \end{enumerate}
The main challenge in our case is to prove that the first step of
the algorithm does not fail. Afterwards, we argue why we can apply
Frieze's results for the remaining step.
%(we will not give details,
%but will claim why the proof given in \cite{Frieze} extends to our
%setting).

\subsection{Almost 5-in 5-out subgraph}

Let $\FIVEINOUT$ be the following subgraph of $D_{m_*}$. For each
vertex $v$, assign a set of neighbors $\OUT(v)$ and $\IN(v)$,
where $\OUT(v)$ are out-neighbors of $v$ and $\IN(v)$ are
in-neighbors of $v$. For saturated and blossomed vertices, $\OUT(v)$
and $\IN(v)$ will be of size 5, and for restricted vertices, they
will be of size 1 (thus $\FIVEINOUT$ is not a 5-in 5-out
directed graph under the strict definition).

Let $E_1$ be the edges of Step I (first $2 n \log \log n$ edges),
and $E_2$ be the edges of Step II (remaining edges).

\begin{itemize}
  \setlength{\itemsep}{1pt} \setlength{\parskip}{0pt}
  \setlength{\parsep}{0pt}
\item If $v$ is saturated, then consider
the first 12 appearances in $E_1$ of $v$ as a first vertex. Some
of these edges might later be used as $\OUT$ or $\IN$ for other
vertices. Hence among these 12 appearances, consider only those
whose second vertex is not in $B_2$. By property (v) of
Definition \ref{def_typicalconfiguration}, there will be at least 11 such
second vertices for a typical configuration. Define $\OUT(v)$ as the first 5 vertices
among them which were directed out from $v$, and $\IN(v)$ as the
first 5 vertices among them which were directed in to $v$ in
$\ALG$.

\item If $v$ blossoms, then consider the first 10 $A$-$B$ edges in $E_2$
connected to $v$, and look at the other end points. Let
$\OUT(v)$ be the first 5 vertices which are an out-neighbor of
$v$ and $\IN(v)$ be the first 5 vertices which are an
in-neighbor of $v$. \end{itemize}

A partially-blossomed vertex, by definition, has $\VTXONEDEG{v}{} +
\VTXFREETWODEG{v}{} + \PHASETWODEG{v}{} \ge 2$, and must fall into
one of the following categories. $(i)$ $\VTXONEDEG{v}{} \ge 2$,
$(ii)$ $\VTXFREETWODEG{v}{} \ge 2$, $(iii)$ $\PHASETWODEG{v}{} \ge
2$, $(iv)$ $\VTXONEDEG{v}{} = 1$, $\VTXFREETWODEG{v}{} = 1$, $(v)$
$\VTXONEDEG{v}{} = 1$, $\PHASETWODEG{v}{} = 1$, and $(vi)$
$\VTXONEDEG{v}{} = 0, \VTXFREETWODEG{v}{} = 1$, $\PHASETWODEG{v}{} =
1$. If it falls into several categories, then pick the first one
among them.

\begin{itemize}
  \setlength{\itemsep}{1pt} \setlength{\parskip}{0pt}
  \setlength{\parsep}{0pt}
\item If $v$ partially-blossoms and $\VTXONEDEG{v}{} \ge 2$,
consider the first two appearances of $v$ in $E_1$ as a first
vertex. The first is an out-edge and the second is an in-edge
(see Section \ref{subsection_orientationrule1}).

\item If $v$ partially-blossoms and $\VTXFREETWODEG{v}{} \ge 2$,
consider the first two appearances of $v$ in $E_1$ as a second
vertex whose first vertex is saturated. The first is an in-edge
and the second is an out-edge (see Section
\ref{subsection_orientationrule1}).

\item If $v$ partially-blossoms and $\PHASETWODEG{v}{} \ge 2$,
consider the first two $A$-$B$ edges in $E_2$ incident to $v$.
One of it is an out-edge and the other is an in-edge. Note that
unlike other cases, the actual order of in-edge and out-edge
will depend on the configuration. But since the configuration
contains all the positions at which $v$ appeared in the process,
the choice of in-edge or out-edge only depends on the
configuration and not on the non-revealed vertices (note that this
is slightly different from the blossomed vertices).

\item If $v$ partially-blossoms and $\VTXONEDEG{v}{} = 1$, $\VTXFREETWODEG{v}{} = 1$,
consider the first appearance of $v$ in $E_1$ as a first vertex,
and the first appearance of $v$ in $E_1$ as a second vertex
whose first vertex is saturated. The former is an out-edge and
the latter is an in-edge.

\item If $v$ partially-blossoms and $\VTXONEDEG{v}{} = 1$, $\PHASETWODEG{v}{} = 1$,
consider the first appearance of $v$ in $E_1$ as a first vertex,
and the first $A$-$B$ edge connected to $v$ in $E_2$. The former
is an out-edge and the latter is an in-edge (see rule 1 in
Section \ref{subsection_orientationrule2}).

\item If $v$ partially-blossoms and $\VTXONEDEG{v}{} = 0, \VTXFREETWODEG{v}{} = 1$, $\PHASETWODEG{v}{} = 1$,
consider the first appearance of $v$ in $E_1$ as a second vertex
whose first vertex is saturated, and the first $A$-$B$ edge
connected to $v$ in $E_2$. The former is an in-edge and the
latter is an out-edge (see rule 2 in Section
\ref{subsection_orientationrule2}). Thus we can construct
$\OUT(v)$ and $\IN(v)$ of size 1 each, for all
partially-blossomed vertices.

\item If $v$ is a bud, then consider the first (and only) $A$-$B$ edge
connected to $v$. Let this edge be $e_s$. For a typical
configuration, by property (iii) of Definition
\ref{def_typicalconfiguration}, we know that $v$ has a neglected
edge connected to it. Let $e_t$ be the first neglected edge of
$v$. By property (v) of Definition
\ref{def_typicalconfiguration}, we know that the first vertex of
the neglected edge is either in $A$ or $B_1$. According to the
direction of this edge, the direction of $e_s$ will be chosen as
the opposite direction (see rule 3 in Section
\ref{subsection_orientationrule2}). As in the
partially-blossomed case with $\PHASETWODEG{v}{} \ge 2$, the
direction is solely determined by the configuration. Thus we can
construct $\OUT(v)$ and $\IN(v)$ of size 1 each (which is
already fixed once we fix the configuration).
\end{itemize}

This in particular shows that $D_{m_*}$ has minimum in-degree and
out-degree at least 1, which is clearly a necessary condition for
the graph to be Hamiltonian. A crucial observation is that, once we
condition on the random edge process having a fixed typical configuration,
we can determine exactly which edges are going to be used
to construct the graph $\FIVEINOUT$ just by looking at the
configuration.

For a set $X$, let $\RV(X)$ be an element chosen independently and
uniformly at random in the set (consider each appearance of $\RV
(X)$ as a new independent copy).
\begin{PROP} \label{prop_bipdistribution}
Let $V' = V \setminus B_2$. Conditioned on the edge process having a
typical configuration, $\FIVEINOUT$ has the following distribution.
\begin{enumerate}[(i)]
  \setlength{\itemsep}{1pt} \setlength{\parskip}{0pt}
  \setlength{\parsep}{0pt}
\item If $v$ is saturated, then $\OUT(v)$
and $\IN(v)$ are a union of 5 copies of $\RV (V')$.
\item If $v$ blossoms,
then $\OUT(v)$ and $\IN(v)$ are a union of 5 copies of $\RV
(A)$.
\end{enumerate}
\end{PROP}
\begin{proof}
For a vertex $v \in V$, the configuration contains the information
of the time of arrival of the edges that will be used to construct
the set $\OUT(v)$ and $\IN(v)$.

If $v$ is a saturated vertex, then
we even know which edges belong to $\OUT(v)$ and $\IN(v)$ (if there
are no $B_2$ vertices connected to the first 12 appearances of $v$
as a first vertex, then the first five odd appearances of $v$ as a
first vertex will be used to construct $\OUT(v)$, and the first five
even appearances of $v$ as a first vertex will be used to construct
$\IN(v)$). Since the non-revealed vertices are independent random
vertices in $V'$, we know that
$\OUT(v)$ and $\IN(v)$ of these vertices consist of 5 independent
copies of $\RV(V')$.

If $v$ blossoms, then the analysis is similar to that of the
saturated vertices. However, even though the configuration
contains the information of which 10 edges will be used to construct
$\OUT(v)$ and $\IN(v)$, the decision of whether the odd edges or the
even edges will be used to construct $\OUT(v)$ depends on the
particular edge process (this is determined by the orientation rule
at Step I). However, since the other endpoints are independent
identically distributed random vertices in $A$, the distribution of
$\OUT(v)$ and $\IN(v)$ is not be affected by the previous edges, and
is always $\RV(A)$ (this is analogous to the fact that the
distribution of the outcome of a coin flip does not depend on
whether the initial position was head or tail).
\end{proof}

\subsection{A small 1-factor} \label{subsection_small1factor}
The main result that we are going to prove in this section is summarized in the following proposition:

\begin{PROP} \label{prop_fiveinout1factor}
Conditioned on the random edge process having a typical configuration,
there exists $\whp$ a 1-factor of $\FIVEINOUT$ containing at most
$2\log n$ cycles, and in which at least $9/10$ proportion of each
cycle are saturated vertices.
\end{PROP}

Throughout this section, rather than vaguely conditioning on
the process having a typical configuration, we will consider a fixed typical
configuration ${\bf c}$ and condition on the event that the edge
process has configuration ${\bf c}$. Proposition \ref{prop_fiveinout1factor}
easily follows
once we prove that there exists a Hamilton cycle $\whp$ under this
assumption. The reason we do this more precise conditioning
is to fix the sets $A, B, B_1,
B_2$ and the edges incident to vertices of $B_2$
(note that these are determined solely by the configuration).
In our later analysis, it is crucial to have these fixed.

To prove Proposition \ref{prop_fiveinout1factor}, we represent
the graph $\FIVEINOUT$ as a certain bipartite graph in which a perfect
matching corresponds to the desired 1-factor of the original graph
$D_{m^*}$. Then using the edge distribution of $\FIVEINOUT$
given in the previous section, we will show that the bipartite
graph $\whp$ contains a perfect matching.
The proof of Proposition \ref{prop_fiveinout1factor}
will be given at the end after a series of lemmas.

\medskip

Define a new vertex set $V^* = \{ v^* | \, v \in V \}$ as a copy of $V$,
and for sets $X \subset V$, use $X^*$ to denote the set of vertices in $V^*$
corresponding to $X$.
Then, in order to find a 1-factor in $\FIVEINOUT$, define an
auxiliary bipartite graph $\BIP(V,V^*)$ over the vertex set $V \cup
V^*$ whose edges are given as following: for every (directed) edge $(u,v)$ of
$\FIVEINOUT$, add the (undirected) edge $(u, v^*)$ to $\BIP$. Note that
perfect matchings of $\BIP$ has a
natural one-to-one correspondence with
1-factors of $\FIVEINOUT$. Moreover, the edge distribution of
$\BIP$ easily follows from the edge distribution of
$\FIVEINOUT$. We will say that $\FIVEINOUT$ is the {\em
underlying directed graph} of $\BIP$. A permutation $\sigma$ of $V^*$ acts on
$\BIP$ to construct another bipartite graph which has
edges $(v, \sigma(w^*))$ for all edges $(v,w^*)$ in $\BIP$.

Our plan is to find a perfect matching which is (almost) a uniform
random permutation, and show that this permutation has at most
$O(\log n)$ cycles (if it were a uniform random permutation, then
this is a well-known result, see, e.g., \cite{ErdTur}). Since our
distribution is not a uniform distribution, we will rely on the
following lemma. Its proof is rather technical, and to avoid
distraction, it will be given in the end of this subsection.

\begin{LEMMA}\label{lem_suff_cod}
Let $X$ be subset of $V$.
Assume that $\whp$, (i) $\BIP$ contains a perfect matching,
(ii) every cycle of the underlying directed graph $\FIVEINOUT$
contains at
least one element from $X$, and (iii) the edge
distribution of $\BIP$ is invariant under arbitrary permutations of
$X^*$. Then $\whp$, there exists a perfect matching which when
considered as a permutation contains at most $2\log n$ cycles.
\end{LEMMA}

%%
%% Hall's theorem
%%
The next set of lemmas establish the fact that $\BIP$ satisfies all
the conditions we need in order to apply Lemma \ref{lem_suff_cod}.
First we prove that $\BIP$ contains a perfect matching. We use the
following version of the well-known Hall's theorem (see, e.g.,
\cite{Diestel}).
\begin{THM} \label{thm_halltheorem}
Let $\Gamma$ be a bipartite graph with vertex set $X \cup Y$ and $|X|=|Y|=n$. If for
all $X' \subset X$ of size $|X'| \leq n/2$, $|N(X')| \ge |X'|$ and for
all $Y' \subset Y$ of size $|Y'| \leq n/2$, $|N(Y')| \ge |Y'|$, then $G$
contains a perfect matching.
\end{THM}

\begin{LEMMA} \label{lemma_bippm}
The graph $\BIP$ contains a perfect matching $\whp$.
\end{LEMMA}
\begin{proof}
We will verify Hall's condition for the graph $\BIP$ to prove the
existence of a perfect matching. Recall that $\BIP$ is a bipartite
graph over the vertex set $V \cup V^*$.

Let us show that every set $D \subset V$ of size $|D| \le n/2$ satisfies $|N(D)| \ge |D|$.
This will be done in two steps. First, if $D \subset B_2$, then this
follows from the fact that $\OUT(v)$ are distinct sets for all $v \in B_2$,
(if they were not distinct, then there will be two restricted vertices which are
at distance 2 away,
and it violates property (v) of Definition \ref{def_typicalconfiguration}).
Second, we prove that for $D \subset V \setminus B_2$,
\[ \left|\, N(D) \cap (V^* \setminus N(B_2)) \,\right| \ge |D|. \]
It is easy to see that the above two facts prove our claim.

Let $D \subset V \setminus B_2$ be a set of size at most $k \le
n/2$. The inequality $|N(D) \cap (V^* \setminus N(B_2)) | < |D|$ can
happen only if there exists a set $N^* \subset V^* \setminus N(B_2)$
such that $|N^*| < k$, and for all $v \in D$ all the vertices of
$\OUT(v)$ belong to $N^* \cup N(B_2)$. Since $D \subset V \setminus
B_2$, every vertex in $D$ has 5 random neighbors distributed
uniformly over some set of size $(1-o(1))n$, and thus the
probability of the above event happening is at most,
\[ k \binom{n}{k}^2 \left(\frac{|N(B_2)| + |N^*|}{(1 - o(1))n}\right)^{5k} \le \left(\frac{e^2 n^2 (\log^{13} n + k)^5}{ k^2 \cdot (1 - o(1))n^5}\right)^k \le \left(\frac{9 (\log^{13}n + k)^5)}{k^2n^3}\right)^k. \]
For the range $9n/20 \le k \le n/2$, we will use the following bound
\[ k \binom{n}{k}^2 \left(\frac{\log^{13}n + k}{(1 - o(1))n}\right)^{5k} \le 2^{2n} \left(\frac{1+o(1)}{2}\right)^{9n/4} \le 2^{-n/5}. \]
Summing over all choices of $k$ we get,
\begin{align*}
 &\sum_{k=1}^{n/2} k \binom{n}{k}^2 \left(\frac{\log^{13}n + k}{(1 - o(1))n}\right)^{5k} \\
 \le&
 \sum_{k=1}^{\log^{14}n} \left(\frac{9(\log^{13} n + k)^5}{ k^2 n^3}\right)^k
 + \sum_{k=\log^{14}n }^{9n/20} \left(\frac{9 (\log^{13} n + k)^5}{ k^2 n^3}\right)^k
 + \sum_{k=9n/20}^{n/2} 2^{-n/5}
 \\
 \le& \sum_{k=1}^{\log^{14}n} \left(\frac{10 \log^{70} n}{n^3}\right)^k  +
\sum_{k=\log^{14}n}^{9n/20} \left(\frac{10 k^3}{ n^3 }\right)^k + o(1) = o(1).
\end{align*}
This finishes the proof that $\whp$ $|N(D)| \ge |D|$ for all $D \subset V$ of
size at most $n/2$. Similarly, for sets $D^* \subset V^*$ of size
$|D^*| \le n/2$, using the sets $\IN(v)$ instead of
$\OUT(v)$ we can show that $\whp$ $|N(D^*)| \ge |D^*|$ in
$\BIP$.
\end{proof}

For restricted vertices $v$, the sets $\OUT(v)$ and $\IN(v)$ are of size 1
and are already fixed since we fixed the configuration. Thus the edge corresponding
to theses vertices will be in $\BIP$. Let
\[ \hat{A} = A \setminus (\cup_{v \in B_2} \OUT(v)), \]
and let $\hat{A}^*$ be the corresponding set inside $V^*$ (note that
$\hat{A}$ and $\hat{A}^*$ are fixed sets). This set will
be our set $X$ when applying Lemma \ref{lem_suff_cod}.
We next
prove that every cycle of $\FIVEINOUT$ contains vertices of $\hat{A}$.

%This fact will eventually help prove the second part of Proposition \ref{prop_fiveinout1factor}, and the first part as well.

\begin{LEMMA} \label{lemma_satproportion}
$Whp$, every cycle $C$ of $\FIVEINOUT$ contains at least
$\left\lceil \frac{9}{10}|C| \right\rceil$ vertices of $\hat{A}$.
\end{LEMMA}
\begin{proof}
Recall that by Proposition \ref{prop_bipdistribution}, for vertices
$v \in V \setminus B_2$, the set $\OUT(v)$ and $\IN(v)$ are
uniformly distributed over $V \setminus B_2$, or $A$. Therefore, for a vertex
$w \in B_2$, the only out-neighbor of $w$ is $\OUT(w)$, and the only
in-neighbor is $\IN(w)$ (note that they are both fixed since we
fixed the configuration). Also note that,
\[
|V \setminus \hat{A}| \leq |V \setminus A| + |B_2| \le |B_1| + 2|B_2|
\le \frac{(\log\log n)^{12}}{\log^2 n}n + 2\log^{13} n \le \frac{n}{\log n}.
\]

We want to show that in the graph $\FIVEINOUT$, $\whp$ every cycle
of length $k$ has at most $k/10$ points from $V \setminus \hat{A}$,
for all $k=1,\ldots, n$. Let us compute the expected number of
cycles for which this condition fails and show that it is $o(1)$.
First choose $k$ vertices $v_1, v_2, \cdots, v_k$ (with order) and
assume that $a$ of them are in $B_2$. Then since we already know the
(unique) out-neighbor and in-neighbor for vertices in $B_2$, for the
vertices $v_1, \cdots, v_k$ to form a cycle in that order, we must
fix $3a$ positions ($a$ for the vertices in $B_2$, and $2a$ for the
in-, and out-neighbors of them by property $(v)$ of Definition
\ref{def_typicalconfiguration}). Assume that among the remaining
$k-3a$ vertices, $\ell$ vertices belong to $V \setminus (\hat{A}
\cup B_2)$. Then for there to be at least $\lceil k/10 \rceil$
vertices among $v_1, \cdots, v_k$ not in $\hat{A}$, we must have $3a
+ \ell \ge \lceil k/10 \rceil$. There are at most $3^k$ ways to
assign one of the three types $\hat{A}, B_2$, and $V \setminus
(\hat{A} \cup B_2)$ to each of $v_1, \cdots, v_k$. Therefore the
number of ways to choose $k$ vertices as above is at most
\[ 3^k \cdot n^{k-\ell-3a} |V \setminus \hat{A}|^\ell |B_2|^{a} \le 3^k \cdot n^{k-\ell-3a} \left(\frac{n}{\log n} \right)^\ell \left(\log^{13} n \right)^{a}  \]

There are $k - 2a$ random edges which has to be present in order to
make the above $k$ vertices into a cycle. For all $i \le k-1$, the
pair $(v_i, v_{i+1})$ can become an edge either by $v_{i+1} \in
\OUT(v_i)$ or $v_i \in \IN(v_{i+1})$ (and also for the pair $(v_1,
v_k)$). There are two ways to choose where the edge $\{v_i,
v_{i+1}\}$ comes from, and if both $v_i$ and $v_{i+1}$ are not in
$B_2$, then $\{v_i, v_{i+1}\}$ will become an edge with probability
at most $\frac{5}{(1-o(1))n}$. Therefore the probability of a fixed
$v_1, \cdots, v_k$ chosen as above being a cycle is at most
$2^{k-2a} \left(\frac{5}{(1-o(1))n} \right)^{k-2a}$, and the
expected number of such cycles is at most
\begin{align*}
&2^{k-2a}
\left(\frac{5}{(1-o(1))n} \right)^{k-2a} \cdot 3^k \cdot n^{k-\ell-3a}
\left(\frac{n}{\log n} \right)^\ell \left(\log^{13} n \right)^{a} \\
\le &\left( \frac{\log^{13}n}{n} \right)^{a} \cdot \left( \frac{1}{\log n}\right)^\ell \cdot (30 + o(1))^k \\
\le &\left( \frac{\log^{13}n}{n} \right)^{a} \cdot \left( \frac{1}{\log n}\right)^{\lceil k/10 \rceil - 3a} \cdot (30 + o(1))^k
\le \left( \frac{\log^{16} n}{n} \right)^{a} \cdot \left( \frac{40}{(\log n)^{1/10}}\right)^{k}.
\end{align*}
where we used $3a + \ell \ge \lceil k/10 \rceil$ for the second
inequality. Sum this over $0 \le \ell \le k$ and $0 \le a \le k$ and
we get
\begin{align*}
\sum_{k=1}^{n} \sum_{\ell=0}^{k} \sum_{a=0}^{k} \left( \frac{\log^{16} n}{n} \right)^{a} \cdot \left( \frac{40}{(\log n)^{1/10}}\right)^{k}
= O\left( \sum_{k=1}^{n} (k+1)\left( \frac{40}{(\log n)^{1/10}}\right)^{k}\right) = o(1),
\end{align*}
which proves our lemma.

\end{proof}

The following simple observation is the last ingredient
of our proof.

\begin{LEMMA} \label{lemma_bipinvariant}
The distribution of $\BIP$ is invariant under the action of an
arbitrary permutation of $\hat{A}^*$.
\end{LEMMA}
\begin{proof}
This lemma follows from the following three facts about the
distribution of $\FIVEINOUT$. First, all the saturated vertices have
the same distribution of $\IN$. Second, for the vertices $v \in V
\setminus B_2$, the distribution of $\OUT$ and $\IN$ is uniform over
a set which contains all the saturated vertices (for some vertices
it is $V \setminus B_2$, and for others it is $A$). Third, for the
vertices $v \in B_2$, the set $\OUT(v)$ lies outside $\hat{A}$ by
definition.
Therefore, the action of an arbitrary permutation
of $\hat{A}^*$ does not affect the distribution of $\BIP$.
\end{proof}

\noindent Note that here it is important that we fixed the
configuration beforehand, as otherwise the set $\hat{A}^*$ will
vary, and a statement such as Lemma \ref{lemma_bipinvariant} will
not make sense.

By combining Lemmas \ref{lem_suff_cod}, \ref{lemma_bippm}, \ref{lemma_satproportion}, and \ref{lemma_bipinvariant}, we obtain Proposition \ref{prop_fiveinout1factor}.

\begin{proof}[Proof of Proposition \ref{prop_fiveinout1factor}]
Lemmas \ref{lemma_bippm}, \ref{lemma_satproportion}, and \ref{lemma_bipinvariant} show
that the graph $\BIP$ has all the properties required for the
application of Lemma \ref{lem_suff_cod} (we use $X = \hat{A}$).
Thus we know that $\whp$,
$\FIVEINOUT$ has a 1-factor containing at most $2\log n$ cycles,
and in which at least 9/10 proportion of each cycle are
saturated vertices (second property by Lemma \ref{lemma_satproportion}).
\end{proof}

We conclude this subsection with the proof of Lemma \ref{lem_suff_cod}.

\begin{proof}[Proof of Lemma \ref{lem_suff_cod}]
For simplicity of notation, we use the notation $\mathcal{B}$
for the random bipartite graph $\BIP$.
Note that both a 1-factor over the vertex set $V$ and a perfect matching
of $(V,V^*)$, can be considered as a permutation of $V$.
Throughout this proof we will not distinguish between these
interpretations and treat 1-factors and perfect matchings also as
permutations.

First, let $f$ be an arbitrary function which for every bipartite
graph, outputs one fixed perfect matching in it. Then, given a
bipartite graph $\Gamma$ over the vertex set $V \cup V^*$, let
$\Phi$ be the random variable $\Phi(\Gamma) := \tau^{-1} f(\tau
\Gamma)$, where $\tau$ is a permutation of the vertices $\hat{A}^*$
chosen uniformly at random. Since the distribution of $\mathcal{B}$
and the distribution of $\tau \mathcal{B}$ are the same by condition
$(iii)$, for an arbitrary permutation $\sigma$ of $\hat{A}^*$,
$\Phi$ has the following property,
\begin{align}
\BFP(\Phi(\mathcal{B}) = \phi) &=  \BFP(\tau^{-1} f(\tau \mathcal{B}) = \phi) \stackrel{(*)}{=} \BFP((\tau\sigma)^{-1} f(\mathcal{\tau \sigma B}) = \phi) \nonumber\\&=  \BFP(\tau^{-1} f(\tau \sigma \mathcal{B}) = \sigma \phi)
\stackrel{(*)}{=} \BFP(\tau^{-1} f(\tau \mathcal{B}) = \sigma \phi) =  \BFP(\Phi(\mathcal{B}) = \sigma\phi). \label{eqn_maybe2}
\end{align}
In the $(*)$ steps, we used $(iii)$, and the fact that if $\tau$ is
a uniform random permutation of $\hat{A}^*$, then so is $\tau
\sigma$, and therefore, $\mathcal{B}, \tau \mathcal{B}$, and
$\tau\sigma\mathcal{B}$ all have identical distribution.

Define a map $\Pi$ from the 1-factors over the vertex set $V$ to the
1-factors over the vertex set $\hat{A}$ obtained by removing all the
vertices that belong to $V \setminus \hat{A}$ from every cycle. For
example, a cycle of the form $(x_1 x_2 y_1 y_2 x_3 y_3 x_4)$ will
become the cycle $(x_1 x_2 x_3 x_4)$ when mapped by $\Pi$ (where
$x_1, \ldots, x_4 \in \hat{A}$, and $y_1,y_2,y_3 \in V \setminus
\hat{A}$). Note that if all the original 1-factors contained at
least one element from $\hat{A}$, then the total number of cycles
does not change after applying the map $\Pi$. This observation
combined with condition $(ii)$ implies that it suffices to obtain a
bound on the number of cycles after applying $\Pi$.

Let $\sigma, \rho$ be permutations of the vertex set $\hat{A}^*$. We
claim that for every 1-factor $\phi$ of the vertex set $V$, the
equality $\sigma \cdot \Pi (\phi) = \Pi (\sigma \cdot \phi)$ holds.
This claim together with (\ref{eqn_maybe2}) gives us,
\begin{align*}
\BFP(\Pi (\Phi(\mathcal{B})) = \rho) &= \BFP( \Phi(\mathcal{B}) \in \Pi^{-1}(\rho)) \stackrel{(\ref{eqn_maybe2})}{=} \BFP( \sigma \Phi( \mathcal{B}) \in \Pi^{-1}(\rho) ) = \BFP( \Pi (\sigma \Phi( \mathcal{B})) = \rho ) \\
&= \BFP( \sigma \cdot \Pi(\Phi(\mathcal{B})) = \rho) = \BFP( \Pi (\Phi(\mathcal{B})) = \sigma^{-1} \rho).
\end{align*}
Since $\sigma$ and $\rho$ were an arbitrary permutation of the
vertex set $\hat{A}$, we can conclude that conditioned on there
existing a perfect matching, $\Pi(\Phi(\mathcal{B}))$ has a uniform
distribution over the permutations of $\hat{A}$. It is a well-known
fact (see, e.g., \cite{ErdTur}) that a uniformly random permutation
over a set of size $n$ has $\whp$ at most $2\log n$ cycles. Since
$\mathcal{B}$ $\whp$~contains a perfect matching by condition $(i)$,
it remains to verify the equality $\sigma \cdot \Pi (\phi) = \Pi
(\sigma \cdot \phi)$. Thus we conclude the proof by proving
this claim.

For a vertex $x \in \hat{A}$, assume that the cycle of $\phi$ which
contains $x$ is of the form $(\cdots x y_1 y_2 \cdots y_k x_+
\cdots)$ ($k \ge 0$) for $y_1,\ldots, y_k \in V \setminus \hat{A}$.
Then by definition $\Pi(\phi) (x) = x_+$, and thus $(\sigma \cdot
\Pi(\phi)) (x) = \sigma(x_+)$. On the other hand, since $\sigma$
only permutes $\hat{A}$ and fixes every other element of $V$, we
have $(\sigma \cdot \phi)(x) = \sigma(y_1) = y_1$, and $(\sigma
\cdot \phi)(y_i) = y_{i+1}$ for all $i \le k-1$, and $(\sigma \cdot
\phi)(y_k) = \sigma(x_+)$. Therefore the cycle in $\sigma \cdot
\phi$ which contains $x$ will be of the form $(\cdots x y_1 y_2
\cdots y_k \sigma(x_+) \cdots)$ , and then by definition we have
$(\Pi (\sigma \cdot \phi))(x) = \sigma(x_+)$.
\end{proof}

\subsection{Combining the cycles into a Hamilton cycle}
\label{subsection_combinecycles}

Assume that as in the previous subsection, we started with a fixed
typical configuration ${\bf c}$, conditioned on the edge process having
configuration ${\bf c}$, and found a 1-factor of $\FIVEINOUT$ by
using Proposition \ref{prop_fiveinout1factor}. Since this 1-factor
only uses the edges which have been used to construct the graph
$\FIVEINOUT$, it is independent of the $A$-$A$ edges in Step II that
we did not reveal. Moreover, by the definition of a typical
configuration, there are at least $\frac{1}{3} n \log n$ such edges.
Note that the algorithm gives a random direction to these edges.
So interpret this as receiving
$\frac{1}{3} n \log n$ randomly directed $A$-$A$ edges with repeated edges
allowed.
Then the problem of finding a directed Hamilton
cycle in $D_{m_*}$ can be reduced to the following problem.

\medskip

Let $V$ be a given set and $A$ be a subset of size $(1-o(1))n$.
Assume that we are given a 1-factor over this vertex set,
where at least $9/10$ proportion of each cycle lies in the set $A$.
If we are given $\frac{1}{3}n \log n$ additional $A$-$A$ edges
chosen uniformly at random,
can we find a directed Hamilton cycle?

\medskip

To further simplify the problem, we remove the vertices $V \setminus
A$ out of the picture. Given a 1-factor over the vertex set $V$,
mark in red, all the vertices not in $A$. Pick any red
vertex $v$, and assume that $v_-, v, v_+ \in V$ appear in this
order in some cycle of the given 1-factor. If $v_- \neq v_+$,
replace the three vertices $v_-, v, v_+$ by a new vertex $v'$, where
$v'$ takes as in-neighbors the in-neighbors of $v_-$, and as
out-neighbors, the out-neighbors of $v_+$. We call the above process
as a \emph{compression} of the three vertices $v_-, v, v_+$. A crucial
property of compression is that every 1-factor of the compressed
graph corresponds to a 1-factor in the original graph (with the same
number of cycles). Since a directed Hamilton cycle is also a
1-factor, if we can find a Hamilton cycle in the compressed graph,
then we can also find one in the original graph.

Now for each $v \in V \setminus A$, compress the three vertices
$v_-, v, v_+$ into a vertex $v'$ and mark it red if and only if
either $v_-$ or $v_+$ is a red vertex. This process always decreases
the number of red vertices. Repeat it until there are no red
vertices remaining, or $v_- = v_+$ for all red vertices $v$. As long
as there is no red vertex in a cycle of length 2 at any point of the
process, the latter will not happen. Consider a cycle whose length
was $k$ at the beginning. Since at least 9/10 proportion of each
cycle comes from $A$ and every compression decreases the number of
vertices by 2, at any time there will be at least $(8/10)k$ non-red
vertices, and at most $(1/10)k$ red vertices remaining in the cycle.
Thus if a cycle has a red vertex, then its length will be at least
9, and this prevents length 2 red cycles. So the compressing
procedure will be over when all the red vertices disappear. Note
that since $|V \setminus A| = |B| = o(n)$, the number of remaining
vertices after the compression procedure is over is at least $n-2|B|
=(1-o(1))n$. As mentioned above, it suffices to find a Hamilton
cycle in the graph after the compression process is over.

Another important property of this procedure is related to the
additional $A$-$A$ edges that we are given. Assume that $v$ is the
first red vertex that we have compressed, where the vertices $v_-,
v, v_+$ appeared in this order in some 1-factor. Further assume that
$v_-$ and $v_+$ are not red vertices. Then since the new vertex $v'$
obtained from the compression will take as out-neighbors the
out-neighbors of $v_+$, and in-neighbors the in-neighbors of $v_-$,
we may assume that this vertex $v'$ is a vertex in $A$ from the
perspective of the new $\frac{1}{3}n\log n$ edges that will be
given.

This observation shows that every pair of vertices of the compressed graph
has the same probability of being one of the new $\frac{1}{3}n\log n$
edges. Since the number of vertices reduced by $o(n)$,
only $o(n \log n)$ of the new edges will be lost because of the
compression. Thus $\whp$ we will be given
$(\frac{1}{3}-o(1))n \log n$ new uniform random edges
of the compressed graph.

\begin{THM} \label{lemma_findhc}
For a typical configuration ${\bf c}$, conditioned on the random edge
process having configuration ${\bf c}$, the directed graph
$D_{m_*}$ $\whp$~contains
a Hamilton cycle.
\end{THM}
\begin{proof}
By Proposition \ref{prop_fiveinout1factor}, there exists $\whp$ a perfect
matching of $\BIP$ which corresponds to a 1-factor in $D_{m_*}$
consisting of at most $2 \log n$ cycles. Also, at least $9/10$
proportion of the vertices in each cycle lies in $A$.
After using the compression argument which has been discussed above,
we may assume that we are given a 1-factor over some vertex set of
size $(1-o(1))n$. Moreover, the random edge process contains at least $(\frac{1}{3}-o(1))n\log n$ additional random directed edges (distributed uniformly over that set).
By Theorem \ref{thm_friezephase23} with $L$ being the whole vertex set, we can conclude that $\whp$~the compressed graph
contains a directed Hamilton cycle, and this in turn implies that
$D_{m_*}$ contains a directed Hamilton cycle.
\end{proof}

\begin{COR}
%Let ${\bf e}$ denote the random edge process.
The directed graph
$D_{m_*}$ $\whp$~contains a Hamilton cycle.
\end{COR}
\begin{proof}
Let ${\bf e}$ be a random edge process.
Let $D = D_{m_*}({\bf e})$ and $\HAM$ be the collection of directed
graphs that contain a directed Hamilton cycle. For a configuration
${\bf c}$, denote by ${\bf e} \in {\bf c}$, the event that ${\bf e}$
has configuration ${\bf c}$. If ${\bf e} \in {\bf c}$ for some
typical configuration ${\bf c}$, then we say that ${\bf e}$ is
typical.

By Theorem \ref{lemma_findhc}, we know that for any typical
configuration ${\bf c}$, $\BFP(D \notin \HAM | {\bf e} \in {\bf c})
= o(1)$, from which we know that $\BFP(\{D \notin \HAM\} \cap
\{\textrm{${\bf e}$ is typical}\}) = o(1)$. On the other hand, by
Lemma \ref{lemma_typicalconfiguration} we know that the probability
of an edge process having a non-typical configuration is $o(1)$.
Therefore $\whp$, the directed graph $D$ is Hamiltonian
\end{proof}

\section{Going back to the original process}\label{sec:MainThmMod}

Recall that the distribution of the random edge process is slightly
different from that of the random graph process since it allows
repeated edges and loops. In fact, one can show that at time $m_*$,
the edge process $\whp$ contains at least $\Omega(\log^2n)$ repeated
edges. Therefore, we cannot simply condition on the event that the
edge process does not contain any repeated edges or loops to obtain
our main theorem for random graph processes.
Our next theorem shows that there exists an on-line algorithm
$\ALGPRIME$ which successfully orients the edges of the random graph
process.

\begin{THM}
There exists a randomized on-line algorithm $\ALGPRIME$ which
orients the edges of the random graph process, so that the resulting
directed graph is Hamiltonian $\whp$ at the time at which the
underlying graph has minimum degree 2.
\end{THM}

The algorithm $\ALGPRIME$ will mainly follow $\ALG$ but with a
slight modification. Assume that we are given a random graph process
(call it the underlying process). Using this random graph process,
we want to construct an auxiliary process whose distribution is
identical to the random edge process. Let $t=1$ at the beginning and
$a_t$ be the number of distinct edges up to time $t$ in our
auxiliary process (disregarding loops). Thus $a_1 = 0$. At time $t$,
with probability $(2a_t+n)/{n^2}$ we will produce a redundant edge,
and with probability $1- (2a_t+n)/{n^2}$, we will receive an edge
from the underlying random graph process. Once we decided to produce
a redundant edge, with probability $2a_t/(2a_t+n)$ choose uniformly
at random an edge out of the $a_t$ edges that already appeared, and
with probability $n/(2a_t+n)$ choose uniformly at random a loop. Let
$e_t$ be the edge produced at time $t$ (it is either a redundant
edge, or an edge from the underlying process), and choose its first
vertex and second vertex uniformly at random. One can easily check
that the process $(e_1, e_2, \cdots, )$ has the same distribution as
the random edge process.

In the algorithm $\ALGPRIME$, we feed this new auxiliary process
into the algorithm $\ALG$ and orient the edges accordingly. Since
the distribution of the auxiliary process is the same as that of the
random edge process, $\ALG$ will give an orientation which $\whp$
contains a directed Hamilton cycle. However, what we seek for is
a Hamilton cycle with no redundant edge.
Thus in the edge process, whenever we see a redundant edge
that is a repeated edge (not a loop),
color it by blue. In order to show that $\ALGPRIME$ gives a
Hamiltonian graph $\whp$, it suffices to show that we can find a
Hamilton cycle in $D_{m_*}$ which does not contain a blue edge
(note that loops cannot be used in constructing a Hamilton cycle).
We first state two useful facts.

\begin{CLAIM} \label{clm_norepeatedblue}
$Whp$, there are no blue edges incident to $B$ used in
constructing $\FIVEINOUT$.
\end{CLAIM}
\begin{proof}
The expected number of blue edges incident to $B$ in Step I used in
constructing $\FIVEINOUT$ can be computed by choosing two vertices
$v$ and $w$ and then computing the probability that $v \in B$, and
$(v,w)$ or $(w,v)$ together appears twice among Step I edges. The probability
that $v$ appears as a first vertex exactly $i$ times is
${n\log\log n \choose i} \left(\frac{1}{n} \right)^i \left(1-\frac{1}{n}\right)^{n\log\log n - i}$.
Condition on the event that $v$ appeared $i$ times as a first vertex
for some $i < 12$ (and also reveal the $i$ positions in which $v$ appeared).
We then compute the probability that some two Step I edges are $(v,w)$ or $(w,v)$.
There are three events that we need to consider.
First is the event that $(v,w)$ appears twice, whose probability
is ${i \choose 2}\left(\frac{1}{n}\right)^2$.
Second is the event that $(v,w)$ appears once and $(w,v)$ appears once,
whose probability is at most ${n \log \log n \choose 1}\frac{1}{n(n-1)}\cdot {i \choose 1}\frac{1}{n}$.
Third is the event that $(w,v)$ appears twice, whose probability
is at most ${n \log \log n \choose 2} \left(\frac{1}{n(n-1)}\right)^2$.
Combining everything, we see that the expected number of Step I
blue edges incident to $B$ is at most,
\begin{align*}
n^2 \cdot \sum_{i=0}^{11} &{n\log \log n \choose i} \left(\frac{1}{n} \right)^{i} \left(1 - \frac{1}{n}\right)^{n\log \log n - i} \times \\
& \left(
{i \choose 2}\left(\frac{1}{n}\right)^2
+ {n \log \log n \choose 1}\left(\frac{1}{n(n-1)}\right) {i \choose 1} \frac{1}{n}
+  {n \log \log n \choose 2} \left( \frac{1}{n(n-1)}\right)^2
\right). \end{align*}
The main term comes from $i=11$, and the third term in the final bracket. Consequently,
we can bound the expectation by
\begin{align*}
(1+o(1)) \cdot n^2 \cdot {n\log \log n \choose 11} \left(\frac{1}{n} \right)^{11} \left(1 - \frac{1}{n}\right)^{n\log \log n - 11} \cdot {n \log \log n \choose 2} \left(\frac{1}{n(n-1)}\right)^2 = o(1).
\end{align*}
We then would like to compute the expected number of blue edges
incident to $B$ in Step 2 used in constructing $\FIVEINOUT$.
Condition on the first vertices of the Step I edges so that we can
determine the sets $A$ and $B$. By Claim \ref{clm_sizes_of_A_B}, we
may condition on the event $|B| = O( \frac{(\log \log
n)^{12}}{\log^2 n})$. Fix a vertex $v \in B$, and expose all appearances of $v$ in Step II,
and note that only the first 10 appearances are relevant. By Claim
\ref{clm_bbedges}, it suffices to bound the probability of the event
that there exists a vertex $w \in A$ such that $(v,w)$ or $(w,v)$ appears twice
among the at most 24 Step I edges where $v$ or $w$ are the first
vertices, and the at most 10 Step II edges which we know is going to
be used to construct the $\OUT$ and $\IN$ of the vertex $v$. Therefore the
expectation is
\begin{align*}
|B|\cdot n \cdot \left( \frac{34}{n} \right)^2
= O\left( \frac{(\log \log n)^{12}}{\log^2 n}n^2\right) \cdot \left(\frac{34}{n}\right)^2 = o(1).
\end{align*}
\end{proof}

\begin{CLAIM} \label{clm_numblue}
$Whp$, there are at most $\log n$ blue edges used in
constructing $\FIVEINOUT$.
\end{CLAIM}
\begin{proof}
By Claim \ref{clm_norepeatedblue}, we know that \whp, all the blue
edges used in constructing $\FIVEINOUT$ are incident to $A$.
Therefore it suffices to show that there are at most $\log n$ blue
edges among the Step I edges. The expected number of such edges can
be computed by choosing two vertices $v,w$, and computing the
probability that $(v,w)$ or $(w,v)$ appears twice. Thus is at most
\[ n^2 \cdot {n\log \log n \choose 2} \left(\frac{2}{n^2}\right)^2 = o(\log n). \]
Consequently, by Markov's inequality, we can derive the conclusion.
\end{proof}

\begin{CLAIM} \label{clm_blueindep_global}
$Whp$, each vertex is incident to at most one blue edge.
\end{CLAIM}
\begin{proof}
It suffices to show that there does not exist three distinct vertices
$v,w_1,w_2$ such that both $\{v,w_1\}$ and $\{v,w_2\}$ appear
at least twice. The probability of this event is at most
\[ {n \choose 3} {m_2 \choose 4} \cdot {4 \choose 2} \left(\frac{2}{n^2}\right)^4 = o(1). \]
\end{proof}

Now assume that we found a 1-factor as in Section
\ref{subsection_small1factor}. By Claim \ref{clm_numblue}, $\whp$,
it contains at most $\log n$ blue edges. Then after performing the
compression process given in the beginning of Section
\ref{subsection_combinecycles}, by Claim \ref{clm_norepeatedblue},
the number of blue edges remains the same as before. Therefore, if
we can find a Hamilton cycle in the compressed graph which does not
use any of the blue edges, then the original graph will also have a
Hamilton cycle with no blue edges. Thus our goal now is to combine
the cycles into a Hamilton cycle without any blue edges, by using
the non-revealed $A$-$A$ edges.

In order to do this, we provide a proof of a slightly stronger form
of Theorem \ref{thm_friezephase23} for $L = V$. In fact, it can be
seen that when combined with the compression argument, this special case
of the theorem implies the theorem for general $L$. Note that we have at least
$\frac{n\log n}{3}$ non-revealed $A$-$A$ edges remaining after finding the
1-factor described in the previous paragraph. Note that these edges
cannot create more blue edges in the 1-factor we previously found,
since all the $A$-$A$ edges used so far appears earlier in the
process than these non-revealed edges. We will find a Hamilton cycle
in two more phases. The strategy of our proof comes from that of
Frieze \cite{Frieze}. In the first phase, given a 1-factor consisting
of at most $O(\log n)$ cycles, we use the first half of the
remaining non-revealed $A$-$A$ edges to combine some of the cycles
into a cycle of length $n-o(n)$. In this phase, we repeatedly
combine two cycles of the 1-factor until there exists a cycle of
length $n-o(n)$.

\begin{LEMMA} \label{lem:originalphase2}
$Whp$, there exists a 1-factor consisting of $O(\log n)$ cycles, one of which
is of length $n - o(n)$. Moreover, this 1-factor contains at most
$O(\log n)$ blue edges.
\end{LEMMA}
\begin{proof}
Condition on the conclusion of Claim \ref{clm_numblue}. Then we are
given a 1-factor consisting of at most $c \log n$ cycles and
containing at most $\log n$ blue edges. Our goal is to modify this
1-factor into a 1-factor satisfying the properties as in the
statement. Consider the non-revealed random $A$-$A$ edges we are
given. Since we will use only the first half of these edges, we have
at least $\frac{n \log n}{6}$ random $A$-$A$ edges given uniformly among
all choices. Let $E_N$ be these edges. Partition $E_N$ as $E_0 \cup
E_1 \cup \cdots \cup E_{c \log n}$, where $E_0$ is the first half of
edges, $E_1$ is the next $\frac{1}{2c\log n}$ proportion of edges,
$E_2$ is the next $\frac{1}{2c\log n}$ proportion of edges, and so
on. Thus $|E_0| = \frac{1}{2}|E_N|$ and $|E_1| = \cdots = |E_{c\log
n}| = \frac{1}{2c\log n}|E_N|$. Since $|E_0| \ge \frac{n \log
n}{12}$, by applying Chernoff's inequality and taking the union
bound, we can see that $\whp$, for every set of vertices $X$ of size
at least $|X| \ge \frac{n}{\log^{1/2}n}$, there exists at least
$\frac{1}{2} |X| |V\setminus X| \frac{\log n}{12n} \ge \frac{n
\log^{1/2} n}{48}$ edges of $E_0$ between $X$ and $V \setminus X$.
Condition on this event.

Assume that the 1-factor currently does not contain a cycle of
length at least $n - \frac{2n}{\log^{1/2}n}$. Then we can partition
the cycles into two sets so that the number of vertices in the
cycles belonging to each part is between $\frac{n}{\log^{1/2}n}$ and
$n - \frac{n}{\log^{1/2}n}$. Thus by the observation above, there
exist at least $\frac{n\log^{1/2}n}{48}$ edges of $E_0$ between the
two parts. Let $(v,w)$ be one such edge. Let $v^+$ be the vertex
that succeeds $v$ in the cycle of the 1-factor that contains $v$,
and let $w^-$ be the vertex that precedes $w$ in the cycle of the
1-factor that contains $w$. If $(w^-,v^+) \in E_1$, then the cycle
containing $v$ and the cycle containing $w$ can be combined into one
cycle (see Figure \ref{fig_rotext}). Therefore, each edge in $E_0$
gives rise to some pair $e$ for which if $e \in E_1$, then some two
cycles of the current 1-factor can be combined into another cycle.
The probability of no such edge being present in $E_1$ is at most
\[ \left( 1 - \frac{1}{n^2} \cdot \frac{n\log^{1/2}n}{48} \right)^{|E_1|}
\le e^{-\Big(\log^{1/2}n/(48n)\Big) \cdot \Big(|E_N|/(2c\log n) \Big)}
\le e^{-\Omega(\log^{1/2}n)}. \]
Therefore with probability $1 - e^{-\Omega(\log^{1/2}n)}$, we can
find an edge in $E_0$ and an edge in $E_1$ which together will reduce
the total number of cycles in the 1-factor by one.

We can repeat the above using $E_i$ instead of $E_1$ in the $i$-th step.
Since the total number of cycles in the initial 1-factor is at most $c\log n$,
the process must terminate before we run out of edges.
Therefore at some step, we must have found a 1-factor that has
at most $O(\log n)$ cycles, and contains a cycle of length $n-o(n)$.
It suffices to check that the estimate on the number of blue edges hold.
Indeed, every time we combine two cycles, we use two additional edges which are not in
the 1-factor, and therefore by the time we are done, we would have
added $O(\log n)$ edges to the initial 1-factor. Therefore
even if all these edges were blue edges, we have $O(\log n)$ blue
edges in the 1-factor in the end.
\end{proof}

\begin{figure}[b]
  \centering
  \begin{tabular}{ccc}
    \input{fig-combinetwo} & \hspace{0.5in} & \input{fig-rotatepath}
  \end{tabular}
  \caption{Combining two cycles, and rotating a path.}
  \label{fig_rotext}
\end{figure}
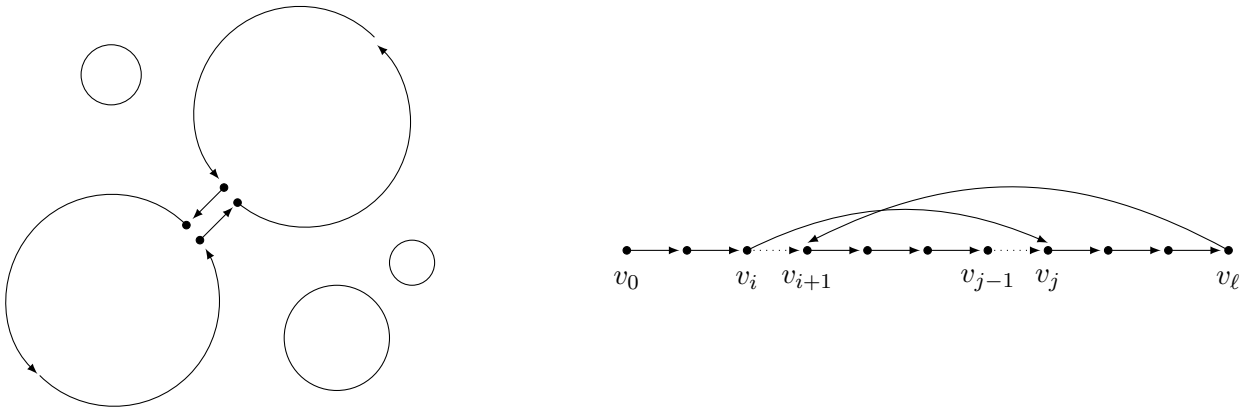

Consider a 1-factor given by the previous lemma. In the second
phase, we use the other half of the remaining new random edges to
prove that the long cycle we just found, can ``absorb'' the
remaining cycles. Let $P=(v_0,\cdots,v_\ell)$ be a path of a
digraph. If there exist two edges $(v_\ell, v_{i+1})$ and $(v_i,
v_j)$ for $1 \le i < \ell$ and $i+1 < j \le \ell$, then we can
\emph{rotate} the path $P$ using $v_{i}$ and $v_{j-1}$ as {\em breaking
points} to obtain a new path $(v_0, v_1, \cdots, v_i, v_j, v_{j+1},
\cdots, v_\ell, v_{i+1}, v_{i+2}, \cdots, v_{j-1})$ (see Figure
\ref{fig_rotext}). We call $v_i$ the {\em intermediate point} of
this rotation. Note that if the graph contains the edge $(v_{j-1},
v_0)$, then one can {\em close} the path into a cycle. Our strategy
is to repeatedly rotate the given path until one can find such an
edge and close the path (see Figure \ref{fig_rotext}).

Further note that the path obtained from $P$ by rotating it once as
above can be described as following. Let $P_1, P_2, P_3$ be subpaths
of $P$ obtained by removing the edges $(v_i, v_{i+1})$ and
$(v_{j-1}, v_j)$. Then there exists a permutation $\pi$ of the set
$[3]$ such that the new path is the path obtained by concatenating
$P_{\pi(1)}, P_{\pi(2)}, P_{\pi(3)}$ (in order). More generally,
assume that we rotate the path $P$ in total $s$ times by using
distinct breaking points $v_{a_1}, v_{a_2}, \cdots, v_{a_{2s}}$. Let
$P_1, \cdots, P_{2s+1}$ be the subpaths of $P$ obtained by removing
the edges $(v_{a_j}, v_{a_j+1})$ for $1 \le j \le 2s$. Then there
exists a permutation $\sigma$ of the set $[2s+1]$ such that the path
we have in the end is the path obtained by concatenating
$P_{\sigma(1)}, P_{\sigma(2)}, \cdots, P_{\sigma(2s+1)}$. We will
use this fact later. Note that it is crucial to have distinct breaking
points here.

After finding a 1-factor described in Lemma \ref{lem:originalphase2},
there are at least $\frac{n\log n}{6}$
non-revealed $A$-$A$ edges that we can use.
Let $E_L$ be the later $\frac{n\log n}{6}$ of these edges, and reveal
all the non-revealed edges not in $E_L$.
Note that there exists a positive constant $C$ such that $\whp$,
the graph induced by the revealed edges before beginning this phase has
maximum degree at most $C\log n$ (it follows from
Chernoff's inequality and union bound). Condition on this event.

We will use the remaining edges $E_L$ in a slightly different way
from how we did in the previous phase since in this phase, it will
be more important to know if some certain edge is present among the
non-revealed edges. For an ordered pair of vertices $e = (x,y)$, let
the {\em flip} of $e$ be $r(e) = (y,x)$ (similarly define a flip of
some set of pairs). Fix some pair $e=(x,y)$, and suppose that we are
interested in knowing whether $e \in E_L$ holds or not, and if $e
\in E_L$, then whether it is a blue edge or not. Thus for each of
the non-revealed edge in $E_L$, ask if it is $e$ or $r(e)$. Since we
know how many times $e$ and $r(e)$ appeared among the already
revealed edges, in the end, we not only know if $e \in E_L$, but
also know if if it is a blue edge or not. We call this procedure as
{\em exposing} the pair $e$, and say that $e$ has been exposed. Note
that the process of exposing the pair $e$ is symmetric in the sense
that even if we are looking only for the edge $e$ we seek for the
existence of $r(e)$ as well. This is because we would like to
determine whether $e$ is blue or not at the same time. We can
similarly define the procedure of exposing a set of pairs, instead
of a single pair. We would like to carefully expose the edges in
order to construct a Hamilton cycle without blue edges.

Note that the expected number of times that $e$ or $r(e)$ appears in
$E_L$ is $\frac{2}{n^2} \cdot \frac{n\log n}{6} = \frac{\log
n}{3n}$. Thus if $S$ is the set of exposed pairs at some point, we
say that the {\em outcome is typical} if the number of times that a
pair belonging to $S$ appears in $E_L$ is at most $\frac{|S|\log
n}{n}$ (which is three times its expected value).
While exposing sets of pairs, we will maintain the outcome
to be typical, since we would like to know that there are enough
non-revealed pairs remaining in $E_L$. For a set $X$ of vertices,
let $Q(X)$ be the set of ordered pairs $(x_1, x_2)$ such that $x_1
\in X$ or $x_2 \in X$.

\begin{LEMMA} \label{lem_removeblue}
Let $X$ and $Y$ be sets of vertices of size at most $\frac{n}{32}$.
Assume that the set of exposed pairs so far is a subset of $Q(X)$ and
the outcome is typical. Further assume that a path $P$ from $v_0$ to
$v_\ell$ of length $\ell = n - o(n)$ is given for some $v_0, v_\ell
\notin X \cup Y$.

Then there exists a set $Z \subset V(P)$ disjoint from $Y$ of size
at most $|Z| \le \frac{n}{\log n \cdot \log \log n}$ such that with
probability at least $1 - o((\log n)^{-1})$, by further exposing
only pairs that intersect $Z$ (thus a subset of $Q(Z)$), one can
find a cycle over the vertices of $P$. Furthermore, the outcome of
exposing these pairs is typical and no new blue edges are added
(thus the set of blue edges in the cycle is a subset of the set of
blue edges in $P$).
\end{LEMMA}

Informally, $Y$ is the set of `reserved' vertices which we would
like to keep non-exposed for later usage. The lemma asserts that we
can close the given path into a cycle by further exposing pairs that
intersect some set $Z$ which is disjoint from $Y$ and has relatively
small cardinality.

\begin{proof}
Denote the path as $P = (v_0, v_1, \cdots, v_\ell)$. For a subset of
vertices $A = \{v_{a_1}, v_{a_2}, \cdots, v_{a_t}\}$, define $A^- =
\{v_{a_1 - 1}, v_{a_2 - 1}, \cdots, v_{a_t -1}\}$ and $A^+ =
\{v_{a_1 + 1}, v_{a_2 + 1}, \cdots, v_{a_t +1}\}$ (if the index
reaches either $-1$ or $\ell + 1$, then we remove the corresponding
vertex from the set).

Our strategy can be described as following. We repeatedly rotate the
path to obtain endpoints, and in each iteration select a set of
vertices and expose only pairs incident to these vertices (call
these vertices as the {\em involved vertices}). Thus a pair
consisting of two non-involved vertices will remain non-exposed. The
set $Z$ will be the set of involved vertices, and our goal will be
to construct a cycle while maintaining $Z$ to be small.

To keep track of the set of vertices that have been involved and the
set of endpoints that we obtained, we maintain two sets $T_i$ and
$S_i$ for $i\ge 0$, where $T_0 = \{v_\ell\}$ and $S_0 = X$.
Informally, $T_i$ will be the set of endpoints that have not yet
been involved, and $S_i$ will be the set of involved vertices while
obtaining the set $T_i$. For example, suppose that we performed a
rotation as in Figure \ref{fig_rotext} in the first round. We will
later see that in the process, we expose the neighbors of $v_\ell$
and $v_i$ for this round of rotation to obtain a new endpoint
$v_{j-1}$. Thus we will add the vertices $v_\ell$ and $v_i$ to $S_1$
and $v_{j-1}$ to $T_1$. It is crucial to maintain $T_i$ as a subset
of the set of non-involved vertices, since we will need to expose
its neighbors in the next round of rotation.

Let $Y_0 = Y \cup \{v_0\}$. Throughout the rotation process, $T_i$ and $S_i$ will satisfy
the following properties:

\begin{enumerate}[(i)]
  \setlength{\itemsep}{1pt} \setlength{\parskip}{0pt}
  \setlength{\parsep}{0pt}
\item for every $w \in T_i$, there exists a path of length $\ell$ from $v_0$ to $w$
whose set of blue edges is a subset of that of $P$,
\item the set of exposed pairs after the $i$-th step is a subset of $Q(S_i)$,
\item all the breaking points used in constructing the paths above belong to $S_i \cup T_i$,
\item $|T_i| = \left( \frac{\log n}{500} \right)^{2i}$ and $|S_i \setminus S_{i-1}| \le 2\left( \frac{\log n}{500} \right)^{2i-1} = \frac{1000}{\log n}|T_i|$ (for $i\ge 1$),
\item $X \cup T_{i-1} \cup S_{i-1} \subset S_i$,
\item $S_i$, $T_i$, and $Y_0$ are mutually disjoint, and
\item the outcome at each iteration is typical.
\end{enumerate}

Recall that $T_0 = \{v_\ell\}$ and $S_0 = X$, and note that the
properties above indeed hold for these sets. Since $S_0 = X$,
property (iv) in particular implies that
\[ |S_i| \le |X| + \sum_{a=1}^{i} \frac{1000}{\log
n}|T_a| \le |X| + \frac{2000}{\log n}|T_i|. \]
 Suppose that we completed constructing the sets
$T_{i}$ and $S_{i}$ for some index $i$ so that $|T_i| \le
\frac{n}{(\log n)^2 \log \log n}$. By (iv), we have $|S_i| \le |X| +
\frac{2000n}{\log n}|T_i| \le (\frac{1}{32} + o(1))n$ and $i =
O(\frac{\log n}{\log \log n})$. We will show how to construct the
sets $T_{i+1}$ and $S_{i+1}$ from these sets.

By $|X| \le \frac{n}{32}$, (ii), (iv)
and (vii), we know that at any step of the process the number of
edges in $E_L$ that remain non-revealed is at least
\[ |E_L| - \frac{|Q(S_i)|\cdot \log n}{n} \ge \frac{n\log n}{6} - \frac{2|S_i|n \log n}{n} \ge \frac{n\log n}{12}. \]
Moreover, the number of non-exposed pairs remaining is at least
\[ n^2 - |Q(S_i)| \ge n^2 - 2n|S_i| \ge \frac{n^2}{2}. \]
We will make use of the following three claims whose proof will be given later.

\begin{CLAIM} \label{clm_blueindep}
Assume that some pairs have been exposed and the outcome is typical.
Once we expose the remaining edges, the probability that there exists
a vertex incident to two new blue edges is at most $o((\log
n)^{-2})$.
\end{CLAIM}

\begin{CLAIM} \label{clm_manyedges}
Assume that some pairs have been exposed and the outcome is typical.
Let $R$ be a set of pairs of size $|R| = \Omega(\frac{n}{\log n})$
disjoint to the exposed pairs. Then with probability at least $1 -
o((\log n)^{-2})$, the number of times a pair in $R$ appear among
the non-revealed edges of $E_L$ is at least $\frac{|R|\log n}{24n}$,
and is at most $\frac{|R| \log n}{2n}$.
\end{CLAIM}

\begin{CLAIM} \label{clm_nottoomanyedges}
Assume that some pairs have been exposed and the outcome is typical.
Then with probability at least $1 - o((\log n)^{-2})$, for every
disjoint sets $A_1, A_2$ of vertices satisfying $|A_1| \le
\frac{n}{\log n \cdot (\log \log n)^{1/2}}$ and $|A_2| =
\frac{|A_1|\log n}{500}$, the number of edges between $A_1$ and
$A_2$ among the non-revealed edges of $E_L$ is at most
$\frac{|A_1|\log n}{100}$.
\end{CLAIM}

For each vertex $w \in T_i$, there exists a path $P_w$ of length
$\ell$ from $v_0$ to $w$ satisfying (i). Let $P_{w,1}$ be the first
half and $P_{w,2}$ be the second half of $P_w$. Let $S_{i+1,0} = S_i
\cup T_i$, and $N = S_{i+1,0} \cup S_{i+1,0}^- \cup Y_0$ and
\[ Q_1 = \{(w, x^+): w \in T_i, x \in V(P_{w,1}) \setminus N \}. \]
We have $Q_1 \subset Q(T_i)$ and
\[ |Q_1| \ge |T_i| \cdot \left(\frac{\ell}{2} - 2|S_i| - 2|T_i| - |Y|-1  \right) \ge \frac{n}{4}|T_i|. \]
By (vi) and the definition of $N$, the pairs in $Q_1$ have both of
their endpoints not in $S_i$, thus have not been exposed yet. Now
expose the set $Q_1$. By Claim \ref{clm_manyedges}, we know that
with probability at least $1 - o((\log n)^{-2})$, the outcome is
typical, and the number of pairs in $Q_1$ that appear in $E_L$ is at
least
\[ \frac{|Q_1|\log n}{24n} \ge \frac{|T_i|\log n}{96}. \] Condition on this event. Note that if
some pair $(w,x^+) \in Q_1$ appears in $E_L$ and is not a blue edge,
then $x$ can serve as an intermediate point in our next round of
rotation. Since we forced not to use the same breaking point twice
by avoiding the set $N$ (see properties (iii) and (v)), if there is
a non-blue edge of the form $(x, y^+)$ for some $y \in P_{w,2}$,
then we can find a path of length $\ell$ from $v_0$ to $y$
satisfying (i) (see Figure \ref{fig_rotext}).

Let
\[ S_{i+1, 1} = \{x :  (w,x^+) \in Q_1 \cap E_L, (w,x^+) \, \textrm{is
not blue} \}. \] By Claim \ref{clm_blueindep}, with probability at
least $1-o((\log n)^{-2})$, among the edges in $Q_1 \cap E_L$, the
number of blue edges is at most $|T_i|$. Condition on this event.
Then the number of non-blue edges between $T_i$ and $S_{i+1,1}^+$ is
at least $\frac{|T_i|(\log n - 1)}{96} >\frac{|T_i|\log n}{100}$. By
Claim \ref{clm_nottoomanyedges}, with probability at least
$1-o((\log n)^{-2})$, we see that $|S_{i+1,1}| \ge \frac{|T_i|\log
n}{500}$. Redefine $S_{i+1,1}$ as an arbitrary subset of it of size
exactly $\frac{|T_i|\log n}{500}$. Note that $S_{i+1,1} \cap N =
\emptyset$. The vertices in $S_{i+1,1}$ will serve as intermediate
points of our rotation.

Now let \[ Q_2 = \{(x,y^+) : x \in S_{i+1,1} ,\,\, (w,x^+)\in Q_1
\cap E_L, \,\, \textrm{and} \,\, y \in V(P_{w,2}) \setminus (N\cup
S_{i+1,1}^-) \}, \] and note that $Q_2 \subset Q(S_{i+1,1})$.
Further note that we are subtracting $S_{i+1,1}^-$ from $V(P_{w,2})$
in the above definition. This is to avoid having both a pair and its
reverse in the set $Q_2$. Even though the set $S_{i+1,1}$ was
defined as a collection of vertices belonging to $P_{w',1}$ for
various choices of $w'$, it can still intersect $P_{w,2}$ for some
vertex $w$, since we are considering different paths for different
vertices. Similarly as before, all the pairs in $Q_2$ are not
exposed yet and we have $|Q_2| \ge \frac{n}{4}|S_{i+1,1}|$.
Moreover, with probability at least $1-o((\log n)^{-2})$, the number
of pairs in $Q_2$ that appear in $E_L$ which are not blue edges is
at least $\frac{|S_{i+1,1}|\log n}{100}$ and the outcome is typical.
Let $T_{i+1,0} = \{y : (x,y^+) \in Q_2 \cap E_L, (x,y^+) \,
\textrm{is not blue} \}$. As in above, with probability at least
$1-o((\log n)^{-2})$, we have $|T_{i+1,0}| \ge \frac{|S_{i+1,1}|\log
n}{500}$. Moreover, by the observation above, for all the vertices
$y \in T_{i+1,0}$, there exists a path of length $\ell$ from $v_0$
to $y$ satisfying (i).

Let $T_{i+1} = T_{i+1,0}$ and $S_{i+1} = S_{i+1,0} \cup S_{i+1,1}$.
Since $|T_{i+1}| \ge \Big(\frac{\log n}{500}\Big)^2|T_i| \ge
\Big(\frac{\log n}{500}\Big)^{2(i+1)}$, we may redefine $T_{i+1}$ as
an arbitrary subset of it of size exactly $\Big(\frac{\log
n}{500}\Big)^{2(i+1)}$. In the previous paragraph we saw that (i)
holds for $T_{i+1}$. Property (ii) holds since the set of newly
exposed pairs is $Q_1 \cup Q_2 \subset Q(S_{i+1,0} \cup S_{i+1,1}) =
Q(S_{i+1})$. Properties (iii), (v), and (vi) can easily be checked
to hold. By Claim \ref{clm_manyedges}, the outcome is typical, and
we have (vii). For property (iv), the size of $T_{i+1}$ by
definition satisfies the bound, and the size of $S_{i+1} \setminus
S_i$ is
\begin{align*}
|S_{i+1} \setminus S_i| &\le |S_{i+1,0} \setminus S_i| + |S_{i+1,1}| \\
&\le |T_i| +  \frac{|T_i|\log n}{500}
\le \left(\frac{\log n}{500}\right)^{2i} \cdot \left(1 + \frac{\log n}{500}\right) \le 2\left(\frac{\log n}{500}\right)^{2i+1}.
\end{align*}

Repeat the above until we reach a set $T_t$ of size $\frac{n}{(\log
n)^2 \cdot \log \log n} \le |T_t| \le \frac{n}{(500)^2 \log \log
n}$. By (iv), we have $t = O(\frac{\log n}{\log \log n})$ and $|S_t|
\le |X| + \frac{n}{125\log n \cdot \log \log n}$. Redefine $T_t$ as
an arbitrary subset of size exactly $\frac{n}{(\log n)^2 \cdot \log
\log n}$. Note that the size of $S_t$ does not necessarily decrease,
and thus we still have $|S_t| \le |X| + \frac{n}{125\log n \cdot
\log \log n}$. We will repeat the process above for the final time
with the sets $S_t$ and $T_t$. This will give $|T_{t+1}| =
\frac{n}{(500)^2\log \log n}$ and $|S_{t+1} \setminus S_{t}| \le
\frac{n}{125 \log n \cdot \log \log n}$, from which it follows that
$|S_{t+1}| \le \frac{2n}{125\log n \log\log n}$. Let $Q_3 = \{(v_0,
z) : z \in T_{t+1}\}$ and expose $Q_3$ (note that the pairs in $Q_3$
has not yet been exposed since $(T_{t+1} \cup \{v_0\}) \cap S_{t+1}
= \emptyset$, while the set of exposed pairs is $Q(S_{t+1})$). Since
$|Q_3| = |T_{t+1}| = \Omega(\frac{n}{\log \log n})$, by Claims
\ref{clm_blueindep} and \ref{clm_manyedges}, with probability at
least $1 - o((\log n)^{-2})$, we have a pair in $Q_3$ that appears
in $E_L$ as a non-blue edge. This gives a cycle over the vertices of
$P$ whose set of blue edges is a subset of that of $P$.

For the set $Z = (S_{t+1} \cup \{v_0\})
\setminus X$, we see that the set of exposed pairs is a subset of $Q(X
\cup Z)$. Furthermore, since $Y_0$ and $S_{t+1}$ are disjoint and $v_0
\notin Y$, the sets $Y$ and $Z$ are disjoint as well. By $t =
O(\frac{\log n}{\log \log n})$, the total number of events involved
is $O(\frac{\log n}{\log \log n})$. Since each event hold with
probability at least $1-o((\log n)^{-2})$, by taking the union
bound, we obtain our set and cycle as claimed with probability at
least $1-o((\log n)^{-1})$.
\end{proof}

The proofs of Claims \ref{clm_blueindep}, \ref{clm_manyedges}, and
\ref{clm_nottoomanyedges} follow.

\begin{proof}[Proof of Claim \ref{clm_blueindep}]
Let $G'$ be the graph induced by the edges that have been
revealed before the final phase (thus all the edges but $E_L$).
It suffices to compute the probability of the following
events: (i) there exist $v,w_1,w_2 \in V$ such that both
$\{v,w_1\}$ and $\{v,w_2\}$ appears at least twice among the
remaining edges, (ii) there exist $v,w_1,w_2 \in V$ such
that $\{v,w_1\}$ and $\{v,w_2\}$ were already in $G'$, and both
appears at least once among the remaining edges, and (iii) there
exist $v, w_1, w_2 \in V$ such that $\{v,w_1\}$ were
already in $G'$, appears at least once among the remaining
edges, and $\{v,w_2\}$ appears at least twice among the remaining
edges.

The probability of the first event happening is at most
\[ n^3 \cdot {n \log n /6 \choose 4} \cdot {4 \choose 2} \left(\frac{2}{n^2}\right)^4 = O\left(\frac{(\log n)^4}{n}\right). \]
Recall that we conditioned on the event that
each vertex has degree at most $C \log n$
in the graph induced by the edges revealed before this phase.
Consequently, the probability of the second event happening is at most
\[ n \cdot {C \log n \choose 2} \cdot {n \log n/6 \choose 2} {2 \choose 1} \cdot \left(\frac{2}{n^2}\right)^2 = O\left(\frac{(\log n)^4}{n}\right), \]
and similarly, the probability of the third event happening is at most
\[ n^2 \cdot {C\log n \choose 1} \cdot {n \log n/6 \choose 3} {3 \choose 1} \cdot \left(\frac{2}{n^2}\right)^3 = O\left(\frac{(\log n)^4}{n}\right). \]
Therefore we have our conclusion.
\end{proof}

\begin{proof}[Proof of Claim \ref{clm_manyedges}]
Recall that at any time of the process, the number of non-revealed edges in $E_L$
is at least $\frac{n\log n}{12}$. The probability of a single non-revealed
edge of $E_L$ being in $R$ is at least $\frac{|R|}{n^2}$. Therefore
the expected number of times a pair in $R$ appear among
the non-revealed edges is at least,
\[ \frac{|R|}{n^2} \cdot \frac{n \log n}{12} = \frac{|R|\log n}{12n}. \]
On the other hand, recall that at any time of the process, the
probability that a non-revealed edge of $E_L$ is some fixed pair at
most $\frac{2}{n^2}$, since the number of non-exposed pairs is at
least $\frac{n^2}{2}$. Therefore the expected number of times a pair
in $R$ appear among the non-revealed edges is at most,
\[ \frac{2|R|}{n^2} \cdot \frac{n \log n}{6} = \frac{|R|\log n}{3n}. \]
Since $|R| = \Omega(\frac{n}{\log \log n})$, the conclusion follows
from Chernoff's inequality and union bound.
\end{proof}

\begin{proof}[Proof of Claim \ref{clm_nottoomanyedges}]
Recall that at any time of the process, the probability that a
non-revealed edge of $E_L$ is $(v,w)$ or $(w,v)$ is at most
$\frac{4}{n^2}$, since the number of non-exposed pairs is at least
$\frac{n^2}{2}$.

Let $k$ be a fixed integer satisfying $k \le \frac{n}{\log n \cdot
\log \log n}$. Let $A_1$ be a set of vertices of size $k$ and $A_2$
be a set of vertices of size $\frac{k\log n}{500}$ disjoint from
$A_1$. The number of choices for such sets is at most
\[ n^k {n \choose k\log n/500} \le
\left( n^{\frac{500}{\log n}} \cdot \frac{500 en}{k \log n} \right)^{k\log n/500}
\le \left( \frac{e^{1000}n}{k \log n} \right)^{k\log n/500}
. \]
The probability of there being more than $\frac{k\log n}{100}$ edges
between $A_1$ and $A_2$ can be computing by first choosing
$\frac{k\log n}{100}$ pairs between $A_1$ and $A_2$, and then
computing the probability that they all appear among the remaining
edges. Thus is at most
\begin{align*}
&{k^2\log n/500 \choose k\log n/100} \cdot \left(\frac{n\log n}{3} \right)^{k \log n/100} \left(\frac{4}{n^2}\right)^{k\log n /100} \\
\le & \left( \frac{ek}{5} \cdot \frac{n\log n}{3} \cdot \frac{4}{n^2} \right)^{k\log n /100} \\
\le & \left( \frac{ 4ek\log n}{15n} \right)^{k\log n /100 }
\le \left( \frac{ k\log n}{n} \right)^{k\log n /100 }.
\end{align*}
Thus by taking the union bound, we see that the probability
of there being such sets $A_1$ and $A_2$ is at most
\[ \sum_{k=1}^{n/(\log n \cdot \log \log n)} \left(\frac{e^{1000} n}{k \log n}\right)^{k\log n / 500} \cdot \left( \frac{ k\log n}{n} \right)^{k\log n /100}
\le \sum_{k=1}^{n/(\log n \cdot \log \log n)} \left(\frac{e^{1000} k^4 \log^4 n}{n^4}\right)^{k\log n / 500}.
\]
Since the summand is maximized at $k=1$ in the range $1 \le k \le
\frac{n}{\log n \cdot \log \log n}$, we see that the right hand side
of above is $o((\log n)^{-2})$.
\end{proof}

We now can find a Hamilton cycle without any blue edges,
and conclude the proof that $\ALGPRIME$ succeeds $\whp$.

\begin{THM} \label{thm_rotationextension}
There exists a Hamilton cycle with no blue edges $\whp$.
\end{THM}
\begin{proof}
By Proposition \ref{prop_fiveinout1factor}, we $\whp$ can find a
1-factor, which by Claims \ref{clm_numblue} and
\ref{clm_blueindep_global} contains at most $\log n$ blue edges that
are vertex-disjoint. By Claim \ref{clm_norepeatedblue}, it suffices
to find a Hamilton cycle after compressing the vertices in $B$ from
the 1-factor, since $\whp$ there are no blue edges incident to $B$.
With slight abuse of notation, we may assume that the compressed
graph contains $n$ vertices, and that we are given at least
$\frac{n\log n}{3}$ random edges over this 1-factor. By Lemma
\ref{lem:originalphase2}, by using half of these random edges, we
can find a 1-factor consisting of cycles $C_0, C_1, \cdots, C_t$ so
that $|C_0| = n - o(n)$ and $t = O(\log n)$. Suppose that there are
$k$ blue edges that belong to the 1-factor, for some $k = O(\log
n)$. We still have a set of at least $\frac{n\log n}{6}$
non-revealed edges $E_L$ that we are going to use in Lemma
\ref{lem_removeblue}.

Let $X$ be a set which we will update throughout the process.
Consider the cycle $C_1$. If it contains a blue edge, then remove it
from the cycle to obtain a path $P_1$. Otherwise, remove an
arbitrary edge from $C_1$ to obtain $P_1 = (w_0, w_1, \cdots,
w_{a})$. Expose the set of pairs $\{(w_a, x) : x \in V(C_0),
%\textrm{$x$ is not incident to a blue edge},
(w_a,x) \textrm{ is not exposed}\}$ which is of size at least $|C_0|
- |X| - 2k = n-o(n)$. By Claims \ref{clm_blueindep} and
\ref{clm_manyedges}, with probability at least $1- o((\log
n)^{-2})$, the outcome is typical and there exists at least one
non-blue edge of the form $(w_a, x)$ for some $x \in V(C_0)$.
Condition on this event. Note that the set of exposed pairs is a
subset of $Q(\{w_a\})$, and that this gives a path $P$ over the
vertices of $C_0$ and $P_1$, which starts at $w_0$ and ends at some
vertex in $C_1$ (thus $w_a$ is not a endpoint). Add $w_a$ to the set
$X$, and let $Y_1$ be the set of vertices incident to some blue edge
that belongs to $C_0$ or $P_1$. Note that $X$, $Y_1$ are disjoint,
the set of exposed pairs is a subset of $Q(X)$, and neither of the
two endpoints of $P$ belong to $X \cup Y_1$. By applying Lemma
\ref{lem_removeblue} with $X$ and $Y = Y_1$, with probability at
least $1 - o((\log n)^{-1})$, we obtain a cycle that contains all
the vertices of $C_0$ and $C_1$. Moreover, the pairs we further
exposed will be a subset of $Q(Z_1)$ for some set $Z_1$ of size at
most $\frac{n}{\log n \cdot \log \log n}$. Condition on this event
and update $X$ as the union of itself with $Z_1$. Note that by the
definition of $Y_1$, $X$ does not intersect any blue edge of the new
cycle.

Repeat the above for cycles $C_2, C_3, \cdots, C_t$. At each step,
the success probability is $1-o((\log n)^{-1})$, and the size of $X$
increases by at most $1+\frac{n}{\log n \cdot \log \log n} \le
\frac{2n}{\log n \cdot \log \log n}$. Since $t = O(\log n)$, we can
maintain $X$ to have size $o(n)$, and thus the process above indeed
can be repeated. In the end, by the union bound, with probability $1
- o(1)$, we find a Hamiltonian cycle which has at most $k$ blue
edges. Let $Y$ be the vertices incident to the blue edges that
belong to this Hamilton cycle. Note that $|Y| \le 2k$ and $X \cap Y
= \emptyset$. Remove one of the blue edges $(y,z)$ from the cycle to
obtain a Hamilton path. Apply Lemma \ref{lem_removeblue} with the
sets $X$ and $Y \setminus \{y,z\}$ to obtain another Hamilton cycle
with fewer blue edges. Since the total number of blue edges is at
most $k = O(\log n)$, the blue edges are vertex-disjoint, and the
probability of success is at least $1 - o( (\log n)^{-1})$, after
repeating this argument for all the blue edges in the original
cycle, we obtain a Hamilton cycle with no blue edge.
\end{proof}

\section{Concluding Remarks}

In this paper we considered the following natural question. Consider
a random edge process where at each time $t$ a random edge $(u,v)$
arrives. We are to give an on-line orientation to each edge at the
time of its arrival. At what time $t^*$ can one make the resulting
directed graph Hamiltonian? The best that one can hope for is to
have a Hamilton cycle when the last vertex of degree one disappears,
and we prove that this is indeed achievable $\whp$.

The main technical difficulty in the proof arose from the existence
of bud vertices. These were degree-two vertices that were adjacent
to a saturated vertex in the auxiliary graph $\FIVEINOUT$. Note that
for our proof, we used the method of deferred decisions, not
exposing the end-points of certain edges and leaving them as random
variables. Bud vertices precluded us from doing this naively and
forced us to expose the end-point of some of the edges which we
wanted to keep unexposed (it is not difficult to show that without
exposing these endpoints, we cannot guarantee the bud vertices to
have degree at least 2). If one is willing to settle for an
asymptotically tight upper bound on $t^*$, then one can choose
$t^*=(1+\varepsilon)n\log n/2$, and then for $n=n(\varepsilon)$
sufficiently large there are no bud vertices. Moreover, since for
this range of $t^*$, the vertices will have significantly larger
degree, the orienting rule can also be simplified. While not making
the analysis ``trivial" (i.e., an immediate consequence of the work
in \cite{Frieze}), this will considerably simplify the proof.

\medskip

\noindent {\bf Acknowledgement.} We are grateful to Alan Frieze for
generously sharing this problem with us, and we thank Igor Pak for
reference \cite{ErdTur}. We would also like to thank
the two referees for their valuable comments.

\end{document}

%% file: fig-combinetwo.tex
\begin{tikzpicture}

% Circle
  \draw[-latex] (2.5cm,2.5cm) arc (45:220:1.4cm);
%  \draw[fill=black] (2.5cm,2.5cm) circle(0.5mm);

  \draw[-latex] (0.48cm,0.48cm) -- (0.08cm,0.08cm);
  \draw[fill=black] (0.5cm,0.5cm) circle(0.5mm);
  \draw[fill=black] (0cm,0cm) circle(0.5mm);

  \draw[-latex] (0cm,0cm) arc (45:225:1.4cm);
  \draw[-latex] (-1.95cm,-2cm) arc (225:360:1.4cm) arc (0:30:1.4cm);

  \draw[-latex] (0.18cm,-0.2cm) -- (0.62cm,0.24cm);
  \draw[fill=black] (0.68cm,0.3cm) circle(0.5mm);
  \draw[fill=black] (0.18cm,-0.2cm) circle(0.5mm);

  \draw[-latex] (0.68cm,0.3cm) arc (230:360:1.4cm) arc (0:47:1.4cm);

  \draw (-1cm, 2cm) circle(0.4cm);
  \draw (3cm, -0.5cm) circle(0.3cm);
  \draw (2cm, -1.5cm) circle(0.7cm);

\end{tikzpicture}

%% file: fig-rotatepath.tex
\begin{tikzpicture}

%%
%% Second
%%
\def\xendtwo{8 cm}
\def\yshifttwo{2.2 cm}
   \draw [color=white] (0,0) circle(0.1mm);

   \foreach \x in {0.0,0.8,...,8.8} {
            \draw [fill=black] (\x cm, \yshifttwo) circle (0.5mm);
        }

   \foreach \x in {0.0,0.8} {
            \draw [-latex] (\x cm,\yshifttwo) -- (\x cm + 0.7cm,\yshifttwo);
        }
   \draw [-latex, dotted] (1.6 cm,\yshifttwo) -- (2.3cm,\yshifttwo);

   \foreach \x in {2.4,3.2,...,4.0} {
            \draw [-latex] (\x cm,\yshifttwo) -- (\x cm + 0.7cm,\yshifttwo);
        }
   \draw [-latex, dotted] (4.8 cm,\yshifttwo) -- (5.5cm,\yshifttwo);

   \foreach \x in {5.6,6.4,7.2} {
            \draw [-latex] (\x cm,\yshifttwo) -- (\x cm + 0.7cm,\yshifttwo);
        }

 %          \draw [dashed] (1.5 cm,\yshift) -- (1.9cm,\yshift);
 %          \draw [dashed] (4.5 cm,\yshift) -- (4.9cm,\yshift);

    \draw[-latex] (\xendtwo, \yshifttwo) .. controls (6cm,\yshifttwo + 1.1cm)
            and (4.4cm,\yshifttwo + 1.1cm) .. (2.4cm, \yshifttwo + 0.1cm);
    \draw[-latex] (1.6cm, \yshifttwo) ..controls(3cm, \yshifttwo + 0.7cm)
            and (4.2cm,\yshifttwo + 0.7cm) .. (5.6cm, \yshifttwo + 0.1cm);

    \draw (2.4cm, \yshifttwo - 0.4cm) node {$v_{i+1}$};
    \draw (1.6cm, \yshifttwo - 0.4cm) node {$v_{i}$};

    \draw (5.6cm, \yshifttwo - 0.4cm) node {$v_{j}$};
    \draw (4.8cm, \yshifttwo - 0.4cm) node {$v_{j-1}$};

    \draw (0.0cm, \yshifttwo - 0.4cm) node {$v_{0}$};
    \draw (8.0cm, \yshifttwo - 0.4cm) node {$v_{\ell}$};

\end{tikzpicture}